\documentclass[11pt]{article}

\usepackage{amsmath, amsthm}
\usepackage{amsmath, amsfonts}
\usepackage{amsmath, amssymb}
\usepackage{amsmath}
\usepackage[all]{xy}

\newtheorem{theorem}{Theorem}[subsection]
\newtheorem{definition}[theorem]{Definition}
\newtheorem{lemma}[theorem]{Lemma}
\newtheorem{lemma-definition}[theorem]{Lemma/Definition}
\newtheorem{proposition}[theorem]{Proposition}
\newtheorem{corollary}[theorem]{Corollary}
\newtheorem{remark}[theorem]{\it Remark}

\newtheorem{example}[theorem]{Example}

\newtheorem{notation}[theorem]{Notation}

\numberwithin{equation}{subsection}

\newtheorem{stheorem}{Theorem}[section]
\newtheorem{sdefinition}[stheorem]{Definition}

\newtheorem{slemma-definition}[stheorem]{Lemma/Definition}

\input{epsfig.sty}

 \newcommand{\Azscriptsize}{\mbox{\scriptsize\it A$\!$z}}

\newcommand{\Bun}{\mbox{\it Bun}\,}

 \newcommand{\scriptsizeChow}{\mbox{\scriptsize\it Chow}}

 \newcommand{\Cohfrak}{\mbox{\it ${\frak C}$oh}\,}

\newcommand{\Endsheaf}{\mbox{\it ${\cal E}\!$nd}\,}

\newcommand{\GL}{\mbox{\it GL}}

\newcommand{\Hilb}{\mbox{\it Hilb}\,}
 \newcommand{\scriptsizeHilb}{\mbox{\scriptsize\it Hilb}}

\newcommand{\Image}{\mbox{\it Im}\,}

\newcommand{\Ker}{\mbox{\it Ker}\,}

\newcommand{\Nil}{\mbox{\it Nil}\,}

\newcommand{\Proj}{\mbox{\it Proj}\,}

\newcommand{\Quot}{\mbox{\it Quot}}
 \newcommand{\scriptsizeQuot}{\mbox{\scriptsize\it Quot}}

\newcommand{\Space}{\mbox{\it Space}\,}

\newcommand{\Spec}{\mbox{\it Spec}\,}
 \newcommand{\boldSpec}{\mbox{\it\bf Spec}\,}
 \newcommand{\smallboldSpec}{\mbox{\small\it\bf Spec}\,}

\newcommand{\Supp}{\mbox{\it Supp}\,}
 \newcommand{\scriptsizeSupp}{\mbox{\scriptsize\it Supp}\,}
\newcommand{\Sym}{\mbox{\it Sym}}

\newcommand{\Wilson}{\mbox{\it\tiny Wilson}}

\newcommand{\degree}{\mbox{\it deg}\,}

\newcommand{\length}{\mbox{\it length}\,}

\newcommand{\nc}{\mbox{\scriptsize\it nc}}
 \newcommand{\tinync}{\mbox{\tiny\it nc}}

\newcommand{\pr}{\mbox{\it pr}}
\newcommand{\pt}{\mbox{\it pt}}
\newcommand{\rank}{\mbox{\it rank}\,}

\newcommand{\redscriptsize}{\mbox{\scriptsize\rm red}\,}

\newcommand{\reg}{\mbox{\it reg}\,}
\newcommand{\reldegree}{\mbox{\rm rel.deg}\,}

\newcommand{\singletonscriptsize}{\mbox{\scriptsize\it singleton}\,}

 \newcommand{\smoothtiny}{\mbox{\tiny\it smooth}\,}

\newcommand{\torsionscriptsize}{\mbox{\scriptsize\it torsion}\,}

\newcommand{\precleftarrow}{\vspace{-.2ex}\begin{array}{c}\prec \\[-1.9ex]
                                             \leftarrow \end{array}}

\begin{document}

\enlargethispage{23cm}

\begin{titlepage}

$ $

\vspace{-1cm}

\noindent\hspace{-1cm}
\parbox{6cm}{\small August 2008}\
   \hspace{8cm}\
   \parbox[t]{5cm}{math.AG/yymm.nnnn \\ D(2)$\,$: D1, morphism.}

\vspace{2cm}

\centerline{\large\bf
 Morphisms from Azumaya prestable curves with a fundamental module}
\vspace{1ex}
\centerline{\large\bf
 to a projective variety:
 Topological D-strings as a master object for curves}

\bigskip

\vspace{3em}
\centerline{\large
  Si Li,\hspace{1ex}
  Chien-Hao Liu,\hspace{1ex}
  Ruifang Song,\hspace{1ex}\hspace{1ex}and\hspace{1ex}
  Shing-Tung Yau
}

\vspace{6em}

\begin{quotation}
\centerline{\bf Abstract}

\vspace{0.3cm}

\baselineskip 12pt  
{\small
 This is a continuation of our study of the foundations of D-branes
  from the viewpoint of Grothendieck
  in the region of the related Wilson's theory-space
  where ``branes" are still branes.
 In this work, we focus on D-strings and
  construct the moduli stack of morphisms from Azumaya prestable
  curves $C^{A\!z}$ with a fundamental module ${\cal E}$
  to a fixed target $Y$ of a given combinatorial type.
 Such a morphism gives a prototype for a Wick-rotated D-string
  of B-type on $Y$, following the Polchinski-Grothendieck Ansatz,
  and this stack serves as a ground
  toward a mathematical theory of topological D-string world-sheet
  instantons.
} 
\end{quotation}

\vspace{11em}

\baselineskip 12pt
{\footnotesize
 } 

\bigskip

\noindent {\small
MSC number 2000: 14A22, 81T30; 14D20, 14A20, 81T75.
} 

\bigskip
\baselineskip 10pt
{\scriptsize
\noindent{\bf Acknowledgements.}
 We thank Andrew Strominger and Cumrun Vafa for topic
  course/lecture/conversations that continue to influence
  our understanding.
 S.L.\ thanks in addition
  Xi Yin
   for communicating issues on D-brane moduli from physics point
   of view.
 C.-H.L.\ thanks in addition
  Liang Kong
    for lectures/explanations/insights of D-branes from
    boundary conformal field theory aspect  and
  Eric Sharpe
    for communicating to us a long itemized explanation of
    stringy/brany issues/themes/insights
   that both so generously share with and guide us;
  Miao Li and Edoardo Sernesi for communication on issues in project;
  Wenxuan Lu
   for insightful discussions;
  Nathan Berkovits, Frederik Denef, Sergei Gukov, Kentaro Hori, and
  Yongbin Ruan
   for related talks/discussions/conversations;
  Howard Georgi, Michael Hopkins, Pedram Safari, A.S.,
  Katrin Wehrheim, Lauren Williams
   for topic courses fall 2007- spring 2008;
  Ann Willman for inspiring words and hospitality and
  Ling-Miao Chou for moral support.
 The project is supported by NSF grants DMS-9803347 and DMS-0074329.
} 

\end{titlepage}

\newpage
\begin{titlepage}

$ $

\vspace{12em} 

\centerline{\small\it
 Chien-Hao Liu dedicates this work to Ling-Miao Chou}
\centerline{\small\it
 for her tremendous love}
\centerline{\small\it
 throughout his preparatory brewing decade on D-branes.}

\end{titlepage}

\newpage
$ $

\vspace{-4em}  

\centerline{\sc
 Morphisms from Azumaya Curves with Fundamental Modules and D-Strings}

\vspace{2em}

\baselineskip 13pt  

\begin{flushleft}
{\Large\bf 0. Introduction and outline.}
\end{flushleft}

\begin{flushleft}
{\bf Introduction.}
\end{flushleft}
In [L-Y2], we addressed
 the foundation of D-branes
 in the region of the related Wilson's theory-space
   where ``branes" are still branes
 from Grothendieck's point of view
 at Polchinski ([Po1], [Po2]) and Witten ([Wi]).
In the current work, we
 continue this line of study with focus on D-strings and
 lay down part of the foundations
 to address D-string world-sheet instantons.

After highlighting in Sec.~1 from [L-Y2]
 the Polchinski-Grothendieck Ansatz and
 how it leads to a realization of D-branes as morphisms from
 Azumaya schemes with a fundamental module
 to an open-string target-space(-time),
we give a self-contained re-do of this
 for D-strings on a commutative target-space in Sec.~2 and
 bring out the moduli stack
 ${\frak M}_{\Azscriptsize(g,r,\chi)^f}(Y,\beta)$
 of morphisms from Azumaya prestable curves with a fundamental module
 to a fixed target space $Y$
 of combinatorial type $(g,r,\chi\,|\,\beta)$.
The observation that the related Azumaya-type noncommutative-geometric
 setting in the problem (Sec.~2.1) can be recast back to
 a purely commutative-geometric setting
 makes the related moduli problem more accessible (Sec.~2.2).
We prove then a boundedness property for family of morphisms
 in question using this observation (Sec.~2.3) and
 give another presentation of such morphisms when the target
  is a projective space (Sec.~2.4).

In Sec.~3, we discuss further the moduli stack
 ${\frak M}_{\Azscriptsize(g,r,\chi)^f}(Y,\beta)$
 as a preparation for other parts in the project.
A few terse remarks are given in Sec.~4
 to relate the current work to technical themes in separate works.

\bigskip

\noindent
{\bf Convention.}
 %
 \begin{itemize}
  \item[$\cdot$]
   Mathematicians are referred to
   [Po1], [Po2], and [Jo]
    for basics of D-branes when they are still branes,
   [Do] and [Sh] for an insight of how coherent sheaves
    and complexes thereof
    come to play as supersymmetric solitonic D-branes of B-type,
   [As] for a review of geometric phase of D-branes
    from open-string target-space(-time) aspect, and
   [H-H-P] for D-branes of B-type
    from the $2d$-gauged-linear-sigma-model-with-boundary
    point of view.

  \item[$\cdot$]
   All schemes are Noetherian and of finite type over ${\Bbb C}$.
   Similarly for their morphisms and products.
 \end{itemize}

\bigskip

\begin{flushleft}
{\bf Outline.}
\end{flushleft}
{\small
\baselineskip 11pt  
\begin{itemize}
 \item[1.]
  Polchinski-Grothendieck Ansatz, Azumaya prestable curves,
   supersymmetric D-strings of B-type, and surrogates.

 \item[2.]
  Morphisms from Azumaya prestable curves with fundamental modules
   to a projective variety.
  \vspace{-.6ex}
  \begin{itemize}
   \item[2.1]
    Morphisms, their associated surrogate, and
    a prototype description of topological D-strings.

   \item[2.2]
    Azumaya without Azumaya, morphisms without morphisms.

   \item[2.3]
    Boundedness of morphisms.

   \item[2.4]
    Morphisms from an Azumaya prestable curve to ${\Bbb P}^k$.
  \end{itemize}

 \item[3.]
  The stack ${\frak M}_{A\!z(g,r,\chi)^f}(Y,\beta)$ of morphisms
  with a fundamental module.
  \vspace{-.6ex}
  \begin{itemize}
   \item[3.1]
    D$0$-branes revisited: the stack ${\frak M}^{D0}_r(Y)$
     of D0-branes of type $r$ on $Y$.

   \item[3.2]
    The stack ${\frak M}_{A\!z(g,r,\chi)^f}(Y,\beta)$  and
     its decorated atlas.
  \end{itemize}

 \item[4.]
  Remarks: D-strings as a master object for curves and
           D-string world-sheet instantons.
\end{itemize}
} 

\newpage

\section{Polchinski-Grothendieck Ansatz, Azumaya prestable curves,
         supersymmetric D-strings of B-type, and surrogates.}

We introduce in this section Azumaya prestable curves
  with a fundamental module and
 review the most relevant points in [L-Y2] along the way.
See ibidem for details and references.

\bigskip

\begin{flushleft}
{\bf From Polchinski to Grothendieck:
     Noncommutative structure on D-branes.}
\end{flushleft}
A {\it D-brane} (i.e.\ {\sl D}{\it irichlet} {\it mem}{\sl brane})
 is meant to be a boundary condition for open strings
 in whatever form it may take,
 depending on where we are in the related Wilson's theory-space.
A realization of D-branes that is most related to the current work
 is an embedding $f:X\rightarrow Y$ of a manifold $X$
 into the open-string target space-time $Y$
 with the end-points of open strings being required to lie in $f(X)$.
This sets up a $2$-dimensional Dirichlet boundary-value problem
 from the field theory on the world-sheet of open strings.
Oscillations of open strings with end-points in $f(X)$
 then create various fields on $X$,
 whose dynamics is governed by open string theory.
This is parallel to the mechanism that
 oscillations of closed strings create fields in space-time $Y$,
 whose dynamics is governed by closed string theory.
Cf.~Figure 1-1.

\begin{figure}[htbp]
 \epsfig{figure=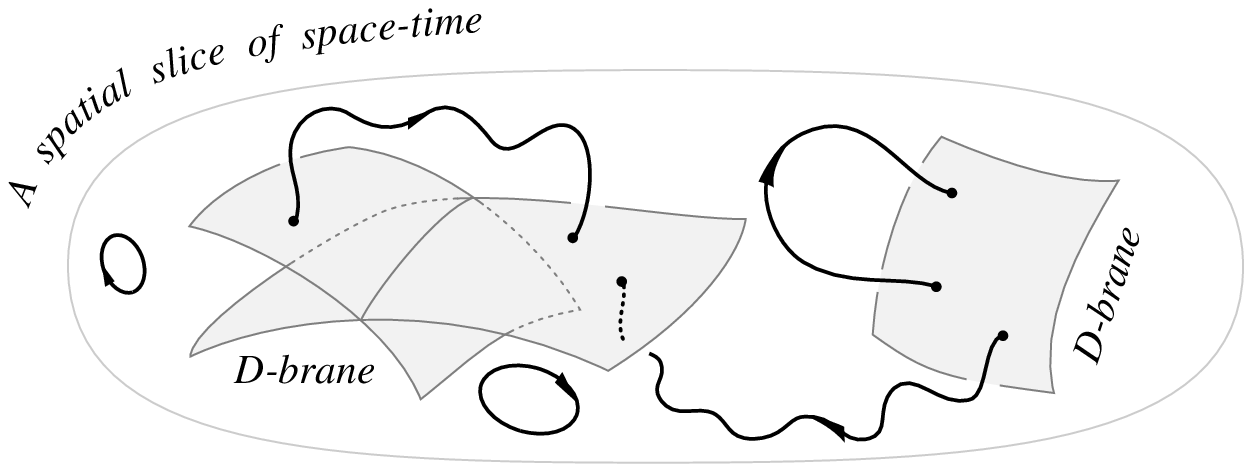,width=16cm}
 \centerline{\parbox{13cm}{\small\baselineskip 12pt
  {\sc Figure} 1-1.
  D-branes as boundary conditions for open strings in space-time.
  This gives rise to interactions of D-brane world-volumes
   with both open strings and closed strings.
  Properties of D-branes,
    including the quantum field theory on their world-volume and
              deformations of such,
   are governed by open and closed strings via this interaction.
  Both oriented open (resp.\ closed) strings and
   a D-brane configuration are shown.
  }}
\end{figure}

\noindent
In this setting, $f$ is realized in local coordinates
 as a tuple of ${\Bbb R}$-valued scalar fields on $X$,
 describing the fluctuations/shapes of the D-brane
 after modding out the longitudinal redundancy along $X$.
When there are $r$-many coincident $D$-branes
 with each of which being described by the same $f$,
 open string theory dictates that
 the components of the tuple that are transverse to $f(X)$
 are enhanced to $M_r({\Bbb R})$-valued.
The meaning of this noncommutative enhancement of such scalar fields
 on coincident D-branes is somewhat mysterious
 from the aspect of space-time itself.
(See the work of Polchinski, e.g.\ [Po1] and [Po2], and
     Witten [Wi] for more details;
  and [L-Y2: references] for a short list of literatures.)

However, when a formal noncommutative extension of Grothendieck's
 local contravariant equivalence of
 (function rings, morphisms) and (geometries, morphisms)
 is applied to the above picture,
the mathematical meaning/content of this enhancement becomes
 the following statement:
 ([L-Y2: Introduction and Sec.\ 2.2]),
 cf.\ Figure~1-2.

\bigskip

\noindent
{\bf Polchinski-Grothendieck Ansatz [D-brane: noncommutativity].}
{\it
 A D-brane (or D-brane world-volume) $X$ carries an Azumaya-type
  noncommutative structure locally associated to a function ring
  of the form $M_r(R)$ for some $r\in {\Bbb Z}_{\ge 1}$ and ring $R$.
 Here, $M_r(R)$ is the $r\times r$ matrix-ring over $R$.
} 

\bigskip

\begin{figure}[htbp]
 \epsfig{figure=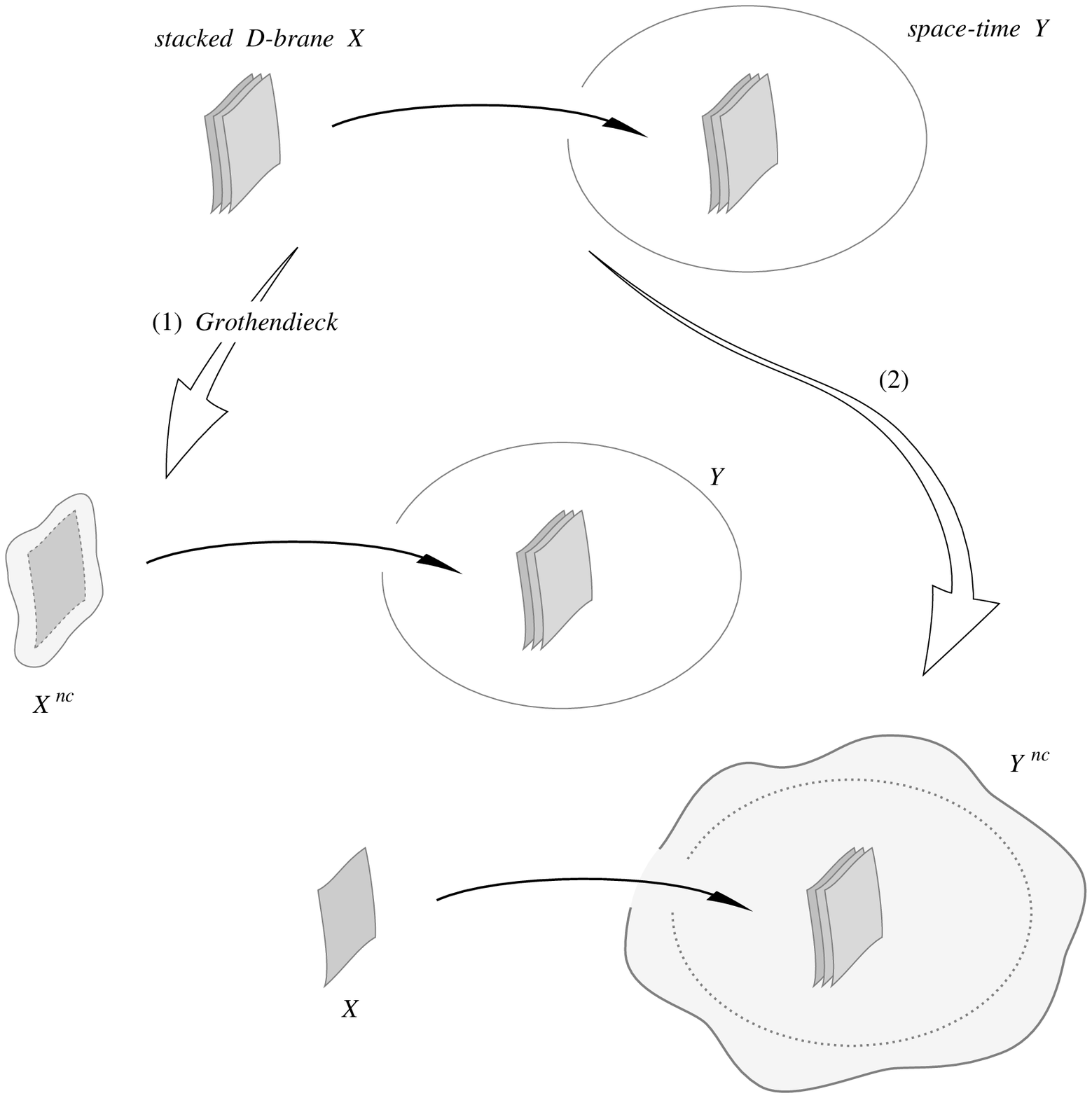,width=16cm}
 \centerline{\parbox{13cm}{\small\baselineskip 12pt
  {\sc Figure} 1-2.
  Two dual/counter aspects on noncommutativity related
   to coincident/stacked D-branes:
   (1) noncommutativity of D-brane world-volume
       as its fundamental/intrinsic nature
   versus
   (2) noncommutativity of space-time as probed by stacked D-branes.
  Aspect (1) leads to the {\it Polchinski-Grothendieck Ansatz} and
   is more fundamental from Grothendieck's viewpoint
   of contravariant equivalence of the category of local geometries
   and the category of function rings.
  }}
\end{figure}

\noindent
An additional statement hidden in this Ansatz that follows from
 mathematical naturality is that
\begin{itemize}
 \item[$\cdot$] {\it
  fields on $X$ are local sections of sheaves ${\cal F}$
   of modules of the structure sheaf ${\cal O}_X^{\nc}$
   of $X$ associated to the above noncommutative structure.}
\end{itemize}
Furthermore, this noncommutative structure on D-branes
 (or D-brane world-volumes)
 is more fundamental than that of space-time in the sense that,
\begin{itemize}
 \item[$\cdot$] {\it
  from Grothendieck's equivalence,
   the noncommutative structure of space-time, if any, can be detected
   by a D-brane only when the D-brane probe itself is noncommutative.}
\end{itemize}
When the closed-string-created B-field on the open-string target
 space-time $Y$ is turned off,
 $R$ in the Ansatz is commutative.
This is the case we will be considering throughout the work.

When D-branes are taken as fundamental objects as strings,
 we no longer want to think of their properties as derived
 from open strings.
Rather, D-branes should have their own intrinsic nature
 in discard of open strings.
Only that when D-branes co-exist with open strings in space-time,
 their nature has to be compatible/consistent with
 the originally-open-string-induced properties thereon.
It is in this sense that we think of a D-brane world-volume
 as an Azumaya-type noncommutative space, following the Ansatz,
 on which other additional compatible structures
  -- in particular, a Chan-Paton module -- are defined.

\bigskip

\begin{flushleft}
{\bf Azumaya prestable curves and
     the world-sheet of SUSY D-strings of B-type.}
\end{flushleft}
The Polchinski-Grothendieck Ansatz for D-branes applies
 to both nonsupersymmetric and supersymmetric D-branes,  and
 to both D-branes of A-type and D-branes of B-type
 in the latter case.
After Wick-rotation,
 the world-volume of D-branes of B-type are holomorphic objects and
 can be studied in the realm of algebraic geometry.
In the picture of Polchinski highlighted in the previous theme,
 the end-points of open strings serve as the source of gauge fields
 on $X$.
The latter correspond to Chan-Paton bundles-with-a-connection on $X$.
In particular,
with the Ansatz and
  a Kobayashi-Hitchin type correspondence/Donaldson-Uhlenbeck-Yau
  type theorem
 in mind,
a mathematical object that can serve as a prototype
 for the (Wick-rotated) world-sheet of a supersymmetric D-string
 of B-type is then given by

\begin{sdefinition}
{\bf [Azumaya prestable curve with a fundamental module].}
{\rm
 (Cf.\ [L-Y2: Definition 1.1.1, Definition 2.2.3].)
 An {\it Azumaya prestable curve} (over ${\Bbb C}$)
   {\it with a fundamental module}
  consists of the triple $(C, {\cal O}_C^{\Azscriptsize}, {\cal E})$,
  where
   $C$ is a nodal curve,
    whose structure sheaf is denoted by ${\cal O}_C$;
   ${\cal O}_C^{\Azscriptsize}$ is a sheaf of noncommutative
    ${\cal O}_C$-algebra, and
   ${\cal E}$ is a (left) ${\cal O}_C^{\Azscriptsize}$-module
    such that, as ${\cal O}_C$-modules,
     ${\cal E}$ is locally-free on $C$ and
     ${\cal O}_C^{\Azscriptsize} = \Endsheaf{\cal E}\;
                                (:=\Endsheaf_{{\cal O}_C}({\cal E}))$.
 The pair $(C,{\cal O}_C^{\Azscriptsize}) =: C^{\Azscriptsize}$
  is called an {\it Azumaya prestable curve} and
  ${\cal E}$ a {\it fundamental module} on $C^{\Azscriptsize}$.
} \end{sdefinition}

Let $r$ be the rank of ${\cal E}$ as an ${\cal O}_C$-module.
Then, here,
 $(C,{\cal O}_C^{\Azscriptsize})$ takes the role of
   the world-sheet of $r$-stacked D-strings of B-type and
 ${\cal E}$ a Chan-Paton module thereon.
When the notion of morphisms from $(C^{\Azscriptsize},{\cal E})$
 to a space-time $Y$ is correctly defined (cf.\ the next theme),
the open-string-induced Higgsing/un-Higgsing behavior of D-strings
 on $Y$ can be correctly reproduced via deformations of such morphisms.
See also Definition~2.1.6 and Remark~2.1.7 in Sec.~2.1.
This is what justifies the setting of [L-Y2] in the end
 as a beginning step to understand D-``branes" and their moduli
 -- a topic that reveals different characters in different regions
  of the related Wilson's theory-space and remains overall
   very mysterious/challenging on the mathematical side
  despite its appearance in the string-theory literature
   [D-L-P] of Dai-Leigh-Polchinski and [Lei] of Leigh
   already in year 1989.

\bigskip

\begin{flushleft}
{\bf The geometry of Azumaya schemes revealed via morphisms therefrom:
     Surrogates.}
\end{flushleft}
Generalization of Grothendieck's theory of schemes to
 the noncommutative case turns out to be a very demanding task.
To bypass this, we take ``function rings" as more fundamental
 than a true ``space with points and a topology".
However, to address the notion of ``gluing local charts on a space",
 one needs a notion of ``localizations of a ring".
This can be done through the notion of {\it Gabriel filters}
 ${\frak F}$ on a ring.
With applications to D-branes in mind, we restrict ourselves to
 a special class of filters that are associated to
 multiplicatively closed subsets in the center $Z(R)$ of rings $R$.
This gives the notion of {\it central localizations} of a ring.
Thus, we define contravariantly
 a ``{\it space}" as an equivalence class $[{\cal S}]$
  of gluing systems ${\cal S}$ of rings
   and
 a {\it morphism} from $\Space[\cal S]$ to $\Space[\cal R]$
  as an equivalence class of systems of $3$-step ring-system-morphisms
  $$
   {\cal R}\; \stackrel{\Phi^{\prime}}{\longrightarrow}\;
     {\cal S}^{\prime\prime}\;
     \precleftarrow\; {\cal S}^{\prime}\;
     \stackrel{\Phi}{\longrightarrow}\; {\cal S}
  $$
  with ${\cal S}^{\prime\prime}$ being a refinement of
  ${\cal S}^{\prime}$ via central localizations.

In defining a morphism
 $\Space[{\cal S}]\rightarrow \Space[{\cal R}]$ as such,
we have in mind
 \begin{itemize}
  \item[(1)]
   in realizing a morphism from $\Space_1$ to $\Space_2$
    as a system of morphisms on local charts,
   the charts on the domain $\Space_1$ in general needs to
    be refined;

  \item[(2)]
   the composition of morphisms $\,\Space_1\rightarrow \Space_2$,
    $\,\Space_2\rightarrow \Space_3\,$ should be a morphism
    $\Space_1\rightarrow \Space_3$.
 \end{itemize}
This would be redundant in the commutative case
 as in that case, the 3-step ring-system-morphism diagram
 ${\cal R}
   \stackrel{\Phi^{\prime}}{\longrightarrow} {\cal S}^{\prime\prime}
   \precleftarrow {\cal S}^{\prime}
                          \stackrel{\Phi}{\longrightarrow} {\cal S}$
 can always be reduced to a $2$-step diagram
 ${\cal R}
   \stackrel{\Phi}{\longrightarrow} {\cal S}^{\prime\prime\prime}
                                           \precleftarrow {\cal S}$.
This reduction no longer holds in general for the noncommutative case
 in our setting.
Thus, the composability of morphisms enforces us to allow
 a morphism to be defined via a medium system ${\cal S}^{\prime}$.
This is how the notion of {\it surrogates}
 (cf.\ $\Space[{\cal S}^{\prime}]$ in the above $3$-step diagram)
 of $\Space[{\cal S}]$ is enforced to occur through the notion
 of morphisms.
It is {\it a compensation for the insufficiency of refinements
 via central localizations}.
It turns out that
{\it surrogates of $\Space[{\cal S}]$ serve also to reveal
     the subtle geometry hidden in $\Space[{\cal S}]$}.

In particular, when these constructions are applied to
 the gluing system of rings associated to
 an Azumaya-type noncommutative space
 $X^{\nc} := (X, {\cal O}_X,{\cal O}_X^{\nc})$,
 where $(X,{\cal O}_X)$ is a commutative (Noetherian) scheme and
       ${\cal O}_X^{\nc}$ is a coherent sheaf of noncommutative
         ${\cal O}_X$-algebras that contains ${\cal O}_X$
         as $1\cdot {\cal O}_X$ in its center
         ${\cal Z}({\cal O}_X^{\nc})$,
one can/should think of $X^{\nc}$
 as ${\cal O}_X^{\nc}$ together with the system
  $L_{{\cal O}_X^{\tinync}}$ of surrogates of $X^{\nc}$
  given by sub-${\cal O}_X$-algebra pairs:
  $$
   L_{{\cal O}_X^{\tinync}}\;
    =\;\left\{\,
         ({\cal A},{\cal A}^{\nc})\;
         \left|\;
         \begin{array}{l}
          {\cal O}_X \subset {\cal A} \subset {\cal A}^{\nc}
                     \subset {\cal O}_X^{\nc}\,;     \\[.6ex]
          \mbox{${\cal A}$, ${\cal A}^{\nc}\,$:
                            sub-${\cal O}_X$-algebras}\,;\,
          {\cal A} \subset {\cal Z}({\cal A}^{\nc})\,
         \end{array}
         \right.
       \right\}\,.
  $$
Cf.\ [L-Y2: Sec.~1.1].
It is this system of surrogates of $X^{\nc}$
 that reveals the very rich geometry hidden in $X^{\nc}$,
allowing us to understand $X^{\nc}$
 without having to directly deal with the technical issue of
 a functorial construction of $\boldSpec({\cal O}_X^{\nc})$
 as a topological space.
Furthermore, any morphism
 $X^{\prime\nc}:=(\boldSpec{\cal A}, {\cal A}, {\cal A}^{\nc})
                                               \rightarrow Y^{\nc}$
 from a surrogate $X^{\prime\nc}$ of $X^{\nc}$
 should be thought of as defining a morphism
 $X^{\nc}\rightarrow Y^{\nc}$ from $X^{\nc}$ itself.
See [L-Y2: Example 1.1.8 and Sec.~4.1] for the simplest example,
 where $X^{\nc} =$ the Azumaya point
  $(\Spec{\Bbb C}, {\Bbb C}, M_r({\Bbb C}))$,
cf.\ Figure~1-3.
\begin{figure}[htbp]
 \epsfig{figure=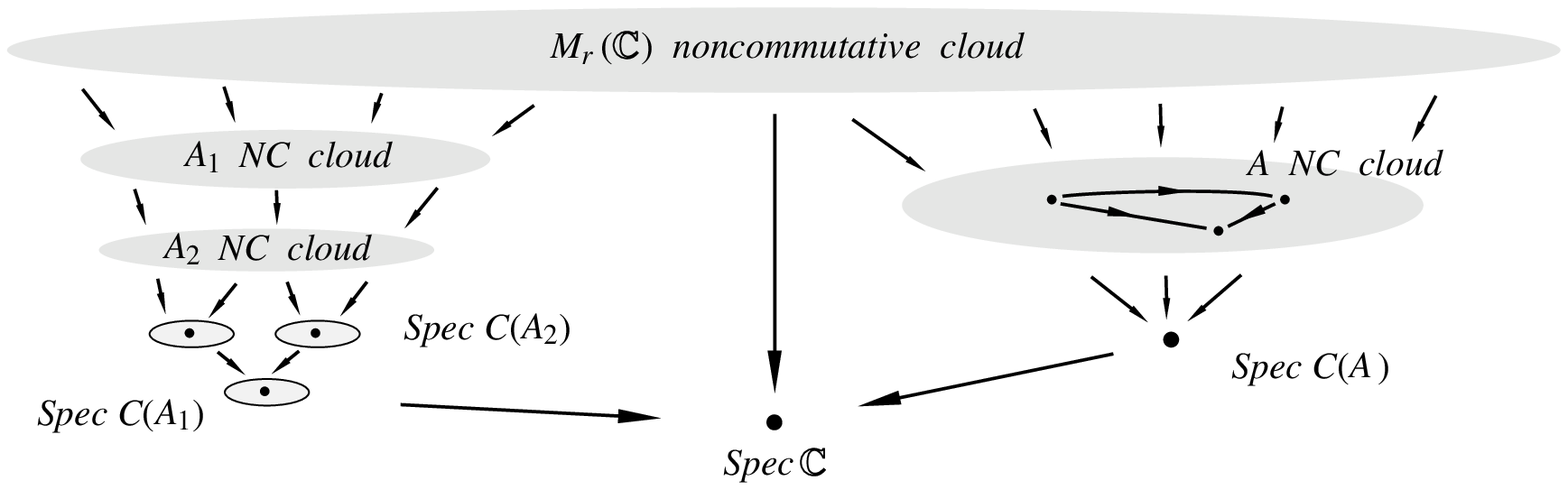,width=16cm}
 \centerline{\parbox{13cm}{\small\baselineskip 12pt
  {\sc Figure} 1-3.
  The very rich geometry of an Azumaya scheme is revealed
   by its surrogates.
  Indicated here is the geometry of an Azumaya point
   $\pt^{\Azscriptsize} := (\Spec{\Bbb C}, {\Bbb C}, M_r({\Bbb C}))$.
  Here, $A_i$ are ${\Bbb C}$-subalgebras of $M_r({\Bbb C})$
    and $C(A_i)$ is the center of $A_i$ with
   $$
     \begin{array}{ccccccc}
      M_r({\Bbb C}) & \supset  & A_1  & \supset  &  A_2
                    &\supset   & \cdots \\
       \cup  && \cup && \cup \\
     {\Bbb C}\cdot {\mathbf 1} & \subset  & C(A_1)
                    & \subset  & C(A_2)   & \subset  & \cdots\,.
     \end{array}
   $$
  From Grothendieck's contravariant equivalence of local geometries
   and function rings, an inclusion pair $R_1\hookrightarrow R_2$
    of ${\Bbb C}$-algebras is equivalent to a dominant morphism
    $\Space R_2\rightarrow \Space R_1$.
  Such a morphism is indicated by a set of $\rightarrow$'s
   in the figure.
  Since, while $M_r({\Bbb C})$ contains only one prime ideal,
    namely $({\mathbf 0})$,
   a ${\Bbb C}$-subalgebra $A\subset M_r({\Bbb C})$
   can have more than one prime ideals,
  it follows that one should think of $\pt^{\Azscriptsize}$
   as containing many secret points hidden in its surrogates.
  By smearing Azumaya points along a scheme,
   one sees also the rich Azumaya geometry in general dimensions.
   }}
\end{figure}

\bigskip

\begin{flushleft}
{\bf D-branes in string theory as morphisms from Azumaya-type nc-space.}
\end{flushleft}
Coherent sheaves on $X^{\nc} = (X,{\cal O}_X, {\cal O}_X^{\nc})$
 are defined to be left ${\cal O}_X^{\nc}$-modules that are coherent
 as ${\cal O}_X$-modules.
The notion of push-forward under a morphism can also be defined.

Once these foundations are laid down,
a (Wick-rotated/Euclidean) D-brane of B-type in an open-string target
 space $Y$ is defined to be
 a morphism $\varphi$ from an Azumaya-type noncommutative space
  $X^{\nc}=(X,{\cal O}_X, {\cal O}_X^{\nc})$ to $Y$
 with the Chan-Paton sheaf the push-forward $\varphi_{\ast}{\cal E}$
  of a fundamental ${\cal O}_X^{\nc}$-module ${\cal E}$ on $X$.

A good match happens: {\it
 The notion of morphisms defined via surrogates},
   which follows from the above mathematical reasonings
   that attempt to extend Grothendieck's language
   of (commutative) schemes to the noncommutative case,
  {\it is exactly what is needed to model/reproduce correctly
   the Higgsing/un-Higgsing behavior of D-branes under deformations.}
It is this miracle that justifies the prototype formulation of
 supersymmetric D-branes of B-type in superstring theory in [L-Y2].
(See [L-Y2: Sec.~4.1] for the simplest example:
  D$0$-branes on the complex line.)
Cf.\ Remark~2.1.7, Figure~2-1-1, and Remark~2.1.16.

For D-branes of A-type and nonsupersymmetric D-branes,
 the philosophy is the same but the language will be different.

\bigskip

When the target $Y$ is a commutative scheme,
 the general formulation of [L-Y2: Sec.~1]
 -- in particular, morphisms and the system of relevant surrogates --
 can be simplified/streamlined.
This is what we will do in Sec.~2.
Such simplification enables us to access the moduli space/stack
 of topological D-strings in $Y$
 along the line of the Polchinski-Grothendieck Ansatz.

\bigskip

\section{Morphisms from Azumaya prestable curves
         with fundamental modules to a projective variety.}

Basic definitions, objects, and properties concerning
 morphisms from Azumaya prestable curves to a projective variety $Y$
 are given in Sec.~2.1.
This follows [L-Y2] and gives a natural noncommutative extension
 of the parallel notions in commutative algebraic geometry (e.g.\ [Ha]).
In the passing, we
 incorporate its use in the description of (Euclidean) D1-branes of
  B-type in superstring theory  and
 bring forth a moduli problem the project is devoted to.
A commutative recast of the noncommutative setting is given in Sec.~2.2.
This recast is technically important and
 it renders the moduli problem more accessible.
Some boundedness properties on the family of morphisms
 from a bounded family of Azumaya prestable curves
  with a fundamental module
 to $Y$ of a fixed combinatorial type are given in Sec.~2.3.
A discussion on presentations of morphisms from
 an Azumaya prestable curve to a projective space is given in Sec.~2.4.

\bigskip

\subsection{Morphisms, their associated surrogate, and
            a prototype description of topological D-strings.}

\begin{flushleft}
{\bf Surrogates, morphisms, and D-strings of B-type.}
\end{flushleft}
\begin{definition} {\bf [(commutative) surrogate].} {\rm
(Cf.\ [L-Y2: Definition 1.1.1 and Definition/Example 1.1.2].)
 Let ${\cal O}_C\subset {\cal A}\subset {\cal O}_C^{\Azscriptsize}$ be
  a commutative ${\cal O}_C$-subalgebra of ${\cal O}_C^{\Azscriptsize}$.
 Then $C_{\cal A}:= \boldSpec {\cal A}$ is called a (commutative)
  {\it surrogate} of $C^{\Azscriptsize}:=(C,{\cal O}_C^{\Azscriptsize})$.
} \end{definition}

\noindent
One should think of $C_{\cal A}$ as a finite scheme over
  and dominating $C$
 that is itself canonically dominated by $C^{\Azscriptsize}$.
An affine cover of $C_{\cal A}$ corresponds to
 a gluing system of algebras from central localizations.
Following this, the notion of morphisms from $C^{\Azscriptsize}$
 to $Y$, as an equivalence class of gluing systems of ring-homomorphisms
  with respect to covers, can be phrased as

\begin{definition} {\bf [morphism].} {\rm
(Cf.\ [L-Y2: Definition 1.1.1].)
 A {\it morphism} from $C^{\Azscriptsize}$ to $Y$,
  in notation $\varphi: C^{\Azscriptsize}\rightarrow Y$,
  is an equivalence class of pairs
  $$
   ({\cal O}_C \subset {\cal A}
               \subset {\cal O}_C^{\Azscriptsize}\;,\;
     f:C_{\cal A}:=\boldSpec{\cal A}\rightarrow Y)\,,
  $$
  where
  \begin{itemize}
   \item[(1)]
    ${\cal A}$ is a commutative ${\cal O}_C$-subalgebra
     of ${\cal O}_C^{\Azscriptsize}$;

   \item[(2)]
    $f:C_{\cal A} \rightarrow Y$
    is a morphism of (commutative) schemes;

   \item[(3)]
    two such pairs
     $({\cal O}_C \subset {\cal A}_1
                \subset {\cal O}_C^{\Azscriptsize}\;,\;
      f_1:C_{{\cal A}_1}\rightarrow Y)$ and
     $({\cal O}_C \subset {\cal A}_2
                \subset {\cal O}_C^{\Azscriptsize}\;,\;
      f_2:C_{{\cal A}_2}\rightarrow Y)$
     are equivalent, in notation
     $$
      ({\cal O}_C \subset {\cal A}_1
                  \subset {\cal O}_C^{\Azscriptsize}\;,\;
       f_1:C_{{\cal A}_1}\rightarrow Y)\;
      \sim\;
      ({\cal O}_C \subset {\cal A}_2
                  \subset {\cal O}_C^{\Azscriptsize}\;,\;
           f_2:C_{{\cal A}_2}\rightarrow Y)\,,
     $$
    if there exists a third pair
     $({\cal O}_C \subset {\cal A}_3
                  \subset {\cal O}_C^{\Azscriptsize}\;,\;
      f_3:C_{{\cal A}_3}\rightarrow Y)$
     such that
      ${\cal A}_3 \subset {\cal A}_i$ and that
      the induced diagram
       \begin{eqnarray*}
       \xymatrix{
        C_{{\cal A}_i}\ar[drr]^{f_i}\ar[d] &&   \\
        C_{{\cal A}_3}\ar[rr]^{f_3}        && Y \\
        }
       \end{eqnarray*}
       commutes, for $i=1,\, 2$.
  \end{itemize}
 To improve clearness, we denote the set of pairs
  associated to $\varphi$ by the bold-faced {\boldmath $\varphi$}.
}\end{definition}

\begin{definition}
{\bf [associated surrogate, canonical presentation, and image].}
{\rm
 Let
  $$
   {\cal A}_{\varphi}\;
    =\; \cap_{({\cal O}_C\subset {\cal A} \subset {\cal O}^{A\!z},
               f:C_{\cal A}\rightarrow Y)
              \in \mbox{\scriptsize\boldmath $\varphi$}}\,
         {\cal A}\,.
  $$
 Then
  ${\cal O}_C\subset {\cal A}_{\varphi}\subset {\cal O}_C^{\Azscriptsize}$
  and there exists a unique
   $f_{\varphi}: C_{\varphi} :=\boldSpec{\cal A}_{\varphi}\rightarrow Y$
  such that the induced diagram
   \begin{eqnarray*}
    \xymatrix{
     C_{\cal A}\ar[drr]^f\ar[d]         &&    \\
     C_{\varphi}\ar[rr]^{f_{\varphi}}   && Y  \\
     }
   \end{eqnarray*}
   commutes,
   for all
    $({\cal O}_C \subset {\cal A} \subset {\cal O}_C^{\Azscriptsize},
      f:C_{\cal A}\rightarrow Y) \in$ {\boldmath $\varphi$}.
 We shall call the pair
  $$
   ({\cal O}_C \subset {\cal A}_{\varphi}
               \subset {\cal O}_C^{\Azscriptsize}\;,\;
     f_{\varphi}:
      C_{\varphi}:= \boldSpec{\cal A}_{\varphi}\rightarrow Y)\,,
  $$
  which is canonically associated to $\varphi$,
  {\it the (canonical) presentation} for $\varphi$.
 The scheme $C_{\varphi}$, which dominates $C$, is called
  the {\it surrogate of $C^{\Azscriptsize}$ associated to $\varphi$}.
 We will denote the built-in morphism $C_{\varphi}\rightarrow C$
  by $\pi_{\varphi}$.
 The subscheme $f_{\varphi}(C_{\varphi})$ of $Y$
  is called the {\it image} of $C^{\Azscriptsize}$ under $\varphi$
  and will be denoted $\Image\varphi$ or $\varphi(C^{\Azscriptsize})$
  interchangeably.
}\end{definition}

\begin{remark} {\it $[$minimal property of $C_{\varphi}$$\,]$.}
{\rm
 By construction,
  \begin{itemize}
   \item[$\cdot$] {\it
    there exists no ${\cal O}_C$-subalgebra
    ${\cal O}_C \subset {\cal A}^{\prime}\subset {\cal A}_{\varphi}$
    such that
     $f_{\varphi}$ factors as the composition of morphisms
     $C_{\varphi}
      \rightarrow \boldSpec{\cal A}^{\prime} \rightarrow Y$}.
  \end{itemize}
 We will call this feature
  the {\it minimal property} of the surrogate
  $C_{\varphi}$ of $C^{\Azscriptsize}$ associated to $\varphi$.
}\end{remark}

\begin{definition}
{\bf [isomorphism between morphisms].}
{\rm
 Two morphisms
    $\varphi_1:(C_1^{\Azscriptsize},{\cal E}_1)\rightarrow Y$ and
    $\varphi_2:(C_2^{\Azscriptsize},{\cal E}_2)\rightarrow Y$
  from Azumaya prestable curves with a fundamental module to $Y$
  are said to be {\it isomorphic}
 if there exists an isomorphism
   $h:C_1 \stackrel{\sim}{\rightarrow} C_2$
   with a lifting
    $\widetilde{h}:
     {\cal E}_1 \stackrel{\sim}{\rightarrow} h^{\ast}{\cal E}_2$
   such that
    \begin{itemize}
     \item[$\cdot$]
      $\widetilde{h}:
       {\cal A}_{\varphi_1}
       \stackrel{\sim}{\rightarrow} h^{\ast}{\cal A}_{\varphi_2}$,

     \item[$\cdot$]
      the following diagram commutes
      \begin{eqnarray*}
       \xymatrix{
         C_{\varphi_2}\ar[drr]^{f_{\varphi_2}}\ar[d]_{\widehat{h}}
                                                          &&      \\
         C_{\varphi_1}\ar[rr]^{f_{\varphi_1}}             && Y\; .\\
       }
      \end{eqnarray*}
    \end{itemize}
  Here,
  we denote the induced isomorphism
    ${\cal O}_{C_1}^{\Azscriptsize} \stackrel{\sim}{\rightarrow}
                             h^{\ast}{\cal O}_{C_2}^{\Azscriptsize}$
    of ${\cal O}_{C_1}$-algebras
    (or ${\cal A}_1 \stackrel{\sim}{\rightarrow} h^{\ast}{\cal A}_2$
        of their respective
        ${\cal O}_{C_{\bullet}}$-subalgebras in question)
    via
     $\widetilde{h}:{\cal E}_1
       \stackrel{\sim}{\rightarrow} h^{\ast}{\cal E}_2$
   still by $\widetilde{h}$ and
  $\widehat{h}:C_{\varphi_2}\stackrel{\sim}{\rightarrow} C_{\varphi_1}$
   is the scheme-isomorphism associated to
  $\widetilde{h}: {\cal A}_{\varphi_1}
    \stackrel{\sim}{\rightarrow} h^{\ast}{\cal A}_{\varphi_2}$.
} \end{definition}

\begin{definition}
{\bf [Chan-Paton module].}
{\rm
 Given a morphism
  $\varphi:
    C^{\Azscriptsize}=(C,{\cal O}_C^{\Azscriptsize}) \rightarrow Y$
  with its canonical presentation
  $({\cal O}_C \subset {\cal A}_{\varphi}
               \subset {\cal O}_C^{\Azscriptsize}\;,\;
     f_{\varphi}:C_{\varphi}\rightarrow Y)$.
 Let ${\cal E}$ be a fundamental ${\cal O}_C^{\Azscriptsize}$-module
  on $C^{\Azscriptsize}$.
 Then ${\cal E}$ is automatically
  a ${\cal O}_{C_{\varphi}}$-module\footnote{By
                                   construction, it is automatically
                                    a left ${\cal O}_{C_{\varphi}}$-module.
                                   We then set the right
                                    ${\cal O}_{C_{\varphi}}$-action on
                                    ${\cal E}$ the same as the left
                                    ${\cal O}_{C_{\varphi}}$-action.},
  in notation, $_{{\cal O}_{C_{\varphi}}}{\cal E}$.
 Define the {\it push-forward} $\varphi_{\ast}{\cal E}$
  of ${\cal E}$ to $Y$ under $\varphi$ by
  $f_{\varphi\,\ast}(_{{\cal O}_{C_{\varphi}}}{\cal E})$.
 It is a coherent ${\cal O}_Y$-module
  supported on
  $\Image\varphi = f_{\varphi}(C_{\varphi})$.
 We will call it also the {\it Chan-Paton module on
  $\varphi(C^{\Azscriptsize})$ associated to ${\cal E}$
  under $\varphi$}.
} \end{definition}

\begin{remark}
{\it $[$topological D-string$\,]$.}
{\rm The morphism
 $\varphi: (C^{\Azscriptsize},{\cal E})\rightarrow Y$
  serves as our prototype-definition for the notion of
  {\it topological D-strings on $Y$} in superstring theory,
  cf.\ [L-Y2: Definition 2.2.3].
 When $Y$ is a target-space for open superstrings,
 as $\varphi$ varies, the endomorphism sheaf
  $\Endsheaf_{{\cal O}_{C_{\varphi}}}(_{{\cal O}_{C_{\varphi}}}{\cal E})$
  on $C_{\varphi}$ also varies.
 This ${\cal O}_{C_{\varphi}}$-module carries the information of
  the gauge group on the D-brane as observed/detected by open strings
  in $Y$.
 This is how the Higgsing/un-Higgsing behavior of the gauge theory
  on the Wick-rotated D-string world-sheet is revealed
  in the current setting
  through deformations of morphisms from Azumaya curves to $Y$.
 The case of D0-branes ([L-Y: Sec.~4]) is indicated in Figure~2-1-1.
 (See also the last theme of this subsection.)
}\end{remark}

\begin{figure}[htbp]
 \epsfig{figure=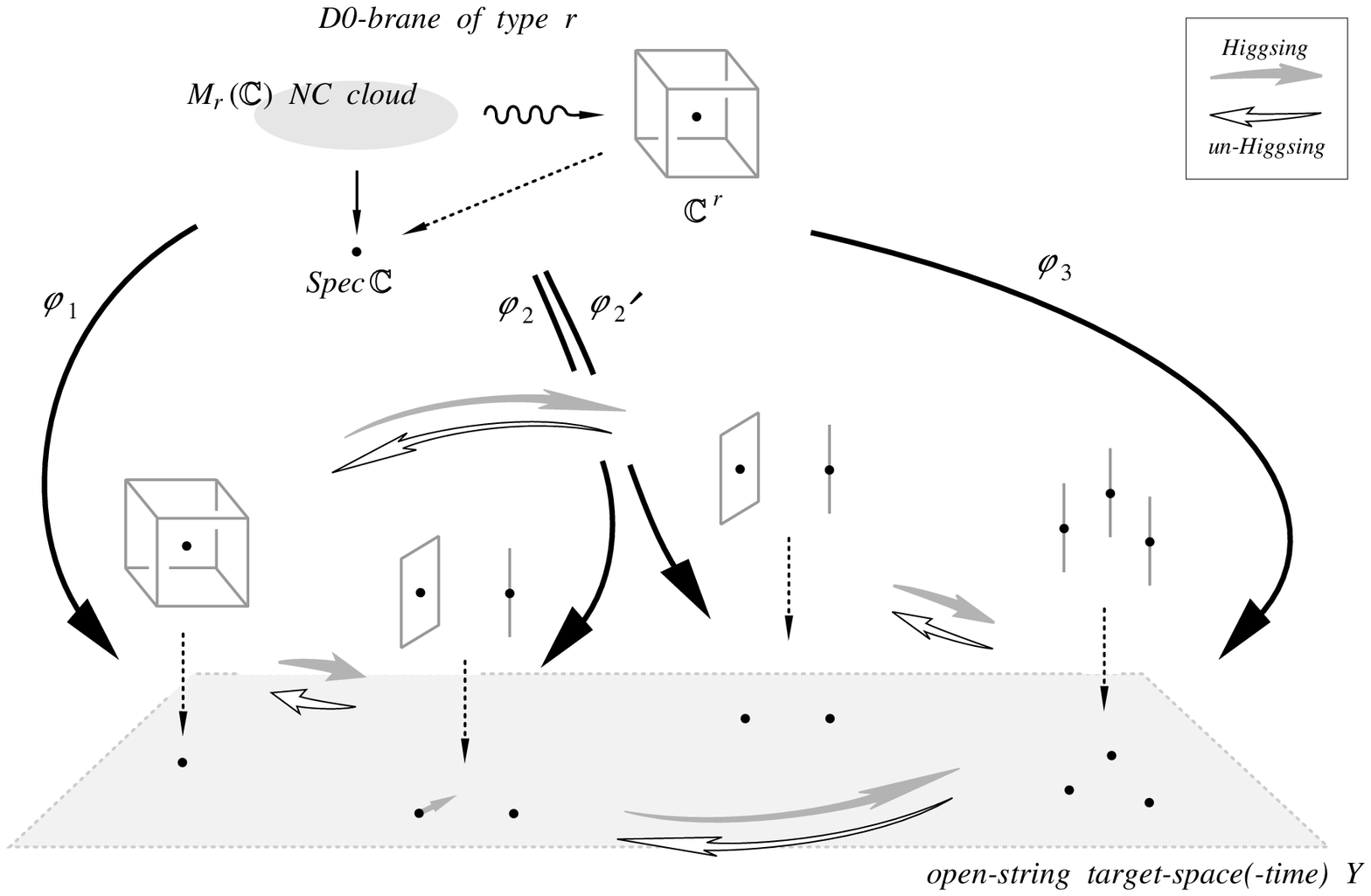,width=16cm}
 \centerline{\parbox{13cm}{\small\baselineskip 12pt
  {\sc Figure} 2-1-1.
  The Higgsing and un-Higgsing of D-branes in $Y$ via deformations of
   morphisms $\varphi$ from $(X^{A\!z},{\cal E})$ to $Y$.
  Here, a module over a scheme is indicated by $\dasharrow$.
  }}
\end{figure}

\bigskip

\begin{flushleft}
{\bf Basic properties of the surrogate $C_{\varphi}$
     associated to $\varphi:C^{\Azscriptsize}\rightarrow Y$.}
\end{flushleft}
The following lemma realizes the surrogate $C_{\varphi}$ of
 $C^{\Azscriptsize}$ associated to
 $\varphi:C^{\Azscriptsize}\rightarrow Y$
 as the graph of a multi-valued morphism from $C$ to $Y$;
it is a direct consequence of the minimal property of $C_{\varphi}$,
 cf.\ Remark~2.1.4:

\begin{lemma-definition}
{\bf [$C_{\varphi}$ as graph of morphism].}
{\it
 With notations from above,
 $C_{\varphi}$ is canonically isomorphic to a subscheme
  $\Gamma_{\varphi}$ of $C\times Y$
  such that
   $\pi_{\varphi}:C_{\varphi}\rightarrow C$ and
   $f_{\varphi}:C_{\varphi}\rightarrow Y$
   are given respectively by
    $\pr_1:\Gamma_{\varphi}\rightarrow C$ and
    $\pr_2:\Gamma_{\varphi}\rightarrow Y$,
   where $\pr_1$ and $\pr_2$ are the projection maps from
    $C\times Y$ to its factors $C$ and $Y$ respectively.
 In particular, $C_{\varphi}$ is projective.}
{\rm
 We shall call $\Gamma_{\varphi}\subset C\times Y$
  the {\it graph} of $\varphi:C^{\Azscriptsize}\rightarrow Y$.
} \end{lemma-definition}

\begin{lemma} {\bf [no embedded point].}
 A surrogate $\boldSpec{\cal A}$ of $C^{\Azscriptsize}$,
  in particular $C_{\varphi}$, does not have embedded points.
\end{lemma}

\begin{proof}
 This follows from the fact that ${\cal A}$ is
  an ${\cal O}_C$-subalgebra of ${\cal O}_C^{\Azscriptsize}$
 and the latter is locally-free -- and hence torsion-free --
  as an ${\cal O}_C$-module.

\end{proof}

\noindent
Thus, $C_{\varphi}$ can be visualized as a twisting-around
 of not necessarily connected and in general non-reduced
 $0$-dimensional schemes over $C$
 that is flat over the smooth locus of $C$.

In more detail\footnote{This
                        applies to general surrogates.
                        Here we focus on $C_{\varphi}$.},
let
 $\nu:\widetilde{C}\rightarrow C$ be the normalization of $C$
  and
 $$
  0\;\longrightarrow\;
      (\nu^{\ast}{\cal A}_{\varphi})_{\torsionscriptsize}\;
     \stackrel{\imath}{\longrightarrow}\;
      \nu^{\ast}{\cal A}_{\varphi}\;
     \stackrel{\jmath}{\longrightarrow}\;
      \overline{\nu^{\ast}{\cal A}_{\varphi}}
     \longrightarrow\; 0\,,
 $$
   where $(\nu^{\ast}{\cal A}_{\varphi})_{\torsionscriptsize}$
    is the torsion subsheaf of $\nu^{\ast}{\cal A}_{\varphi}$,
  and
 $$
   {\cal A}_{\varphi}\; \stackrel{k}{\hookrightarrow}\;
    \nu_{\ast}\overline{\nu^{\ast}{\cal A}_{\varphi}}
 $$
 be the canonical exact sequence (resp.\ inclusion) of
  ${\cal O}_{\widetilde{C}}$- (resp.\ ${\cal O}_C$-)modules.
As ${\cal A}_{\varphi}$ is an ${\cal O}_C$-algebra,
 both $\nu^{\ast}{\cal A}_{\varphi}$
  and $\overline{\nu^{\ast}{\cal A}_{\varphi}}$
 are ${\cal O}_{\widetilde{C}}$-algebra,
 with
  $\boldSpec\nu^{\ast}{\cal A}_{\varphi}
                      =\widetilde{C}\times_C C_{\varphi}$,
  $\imath$ an inclusion of an ideal sheaf on
   $\widetilde{C}\times_C C_{\varphi}$,  and
  $\jmath$ is the ${\cal O}_{\widetilde{C}}$-algebra quotient
   that corresponds to the closed subscheme
  $\widehat{C_{\varphi}}
   :=\boldSpec\overline{\nu^{\ast}{\cal A}_{\varphi}}$
   of $\widetilde{C}\times_C C_{\varphi}$
   obtained by removing all the embedded points from the latter.
By construction,
 $\widehat{C_{\varphi}}$ is flat over $\widetilde{C}$.
The ${\cal O}_C$-algebra inclusion $k$ then implies that
 the morphism
  $\widehat{\nu}:\widehat{C_{\varphi}} \rightarrow C_{\varphi}$
  from the composition
   $\widehat{C_{\varphi}}\hookrightarrow
    \widetilde{C}\times_C C_{\varphi}\rightarrow C_{\varphi}$
  is surjective.
In summary,
 \begin{eqnarray*}
 \xymatrix{
   \widehat{C_{\varphi}}\ar[rr]^{\widehat{\nu}}
                        \ar[d]_{\widehat{\pi_{\varphi}}}
                              ^{\mbox{\scriptsize flat}}
    && C_{\varphi}\ar[d]^{\pi_{\varphi}}       \\
   \widetilde{C}\ar[rr]^{\nu}        &&  C
  }
 \end{eqnarray*}
 where all the arrows/morphisms are surjective.
Furthermore, as ${\cal A}_{\varphi}$ embeds in
 a locally-free ${\cal O}_C$-module,
 for $q$ a node of $C$, let $\nu^{-1}(q)=\{q_-, q_+\}$;
then,
 $$
  (\widehat{\nu}(\widehat{\pi_{\varphi}}^{-1}(q_-)))_{\redscriptsize}\;
  =\;(\pi_{\varphi}^{-1}(q))_{\redscriptsize}\;
  =\;(\widehat{\nu}(\widehat{\pi_{\varphi}}^{-1}(q_+)))_{\redscriptsize}\,.
 $$
Cf.\ Figure~2-1-2.
\begin{figure}[htbp]
 \epsfig{figure=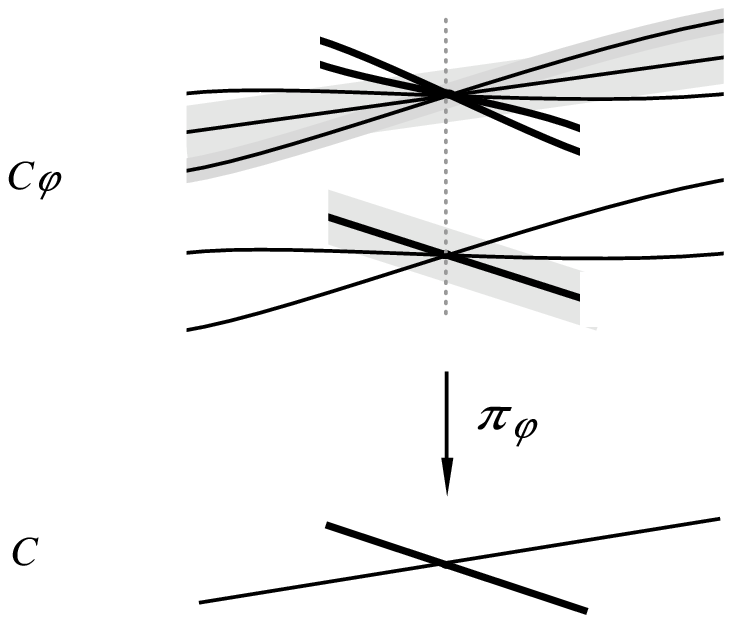,width=16cm}
 \centerline{\parbox{13cm}{\small\baselineskip 12pt
  {\sc Figure} 2-1-2.
  While $C_{\varphi}$ is in general flat only over $C_{\smoothtiny}$,
  the reduced structure of $C_{\varphi}$ over the two branches
   around a node of $C$ has to coincide over that node.
  }}
\end{figure}

Finally, we compare $C_{\varphi}$ in this picture
 with the related twisting-around of (commutative) surrogates
  of the Azumaya ${\Bbb C}$-point
  $\Space M_r({\Bbb C})=(\Spec {\Bbb C}, M_r({\Bbb C}))$
  over $C$.
The inclusion ${\cal A}\hookrightarrow {\cal O}_C^{\Azscriptsize}$
 induces a ${\Bbb C}$-algebra homomorphism
 ${\cal A}\otimes_{{\cal O}_C}\kappa_x\stackrel{\jmath_x}{\rightarrow}
  {\cal O}_C^{\Azscriptsize}\otimes_{{\cal O}_C}\kappa_x
  \simeq M_k({\Bbb C})$, where $\kappa_x$ is the residue field of
  $x\in C$.

\begin{lemma} {\bf [kernel of $\jmath_x$].}
 $\Ker\jmath_x$
  is contained in the nilradical
   $\Nil({\cal A}\otimes_{{\cal O}_C}\kappa_x)$
   of ${\cal A}\otimes_{{\cal O}_C}\kappa_x$  and
 is non-zero only for finitely many $x$ on $C$.
\end{lemma}

\begin{proof}
 Let ${\cal E}$ be a fundamental module on $C^{\Azscriptsize}$
  with ${\cal O}_C^{\Azscriptsize}=\Endsheaf{\cal E}$.
 As ${\cal E}={\pi_{\varphi}}_{\ast}(_{{\cal O}_{C_{\varphi}}}{\cal E})$
  is locally-free of rank $r$,
 $_{{\cal O}_{C_{\varphi}}}{\cal E}$ on $C_{\varphi}/C$
  forms a flat family of $0$-dimensional sheaves of length $r$ over $C$.
 As ${\cal A}_{\varphi}$ is now a subsheaf of $\Endsheaf{{\cal E}}$,
  $\Supp( (_{{\cal O}_{C_{\varphi}}}{\cal E})|_{\pi_{\varphi}^{-1}(x)} )
      = \pi_{\varphi}^{-1}(x)$
  over a dense open subset of $C$.
 This implies that
  $\Supp( (_{{\cal O}_{C_{\varphi}}}{\cal E})|_{\pi_{\varphi}^{-1}(x)} )
    \subset \pi_{\varphi}^{-1}(x)$
  with
  $(\Supp(
     (_{{\cal O}_{C_{\varphi}}}{\cal E})|_{\pi_{\varphi}^{-1}(x)} ))
                                                   _{\redscriptsize}
   = (\pi_{\varphi}^{-1}(x))_{\redscriptsize}$
  for all $x\in C$.
 The lemma now follows.

\end{proof}

The current theme helps us to understand part of the geometry
  behind the notion of ``combinatorial type" of $\varphi$
  in the next theme.
Lemma~2.1.10  
 indicates a general non-flat ``lower-semicontinuous-like"
 dropping behavior of surrogates in a family,
cf.\ Remark~2.1.16.

\bigskip

\begin{flushleft}
{\bf The combinatorial type of a morphism
     $\varphi:(C^{\Azscriptsize},{\cal E}) \rightarrow Y$.}
\end{flushleft}
To define a notion of ``combinatorial type" that
 fits our moduli problem,
 generalizes the situations in several related moduli problems
   from commutative geometry,  and
 captures the features of D-branes in superstring theory,
we have to bring in fundamental modules ${\cal E}$,
 i.e.\ Chan-Paton modules, as well.
The triple $(g,r,\chi)$,
 where $g$ is the (arithmetic) genus of $C$ and
       $r$ and $\chi$ are respectively
        the rank and the Euler characteristic of ${\cal E}$,
 is the standard combinatorial type data for the domain data
  $(C^{\Azscriptsize},{\cal E})$.
It remains to define the notion of ``image curve class"
 for $\varphi:(C^{\Azscriptsize},{\cal E})\rightarrow Y$
 in any of $A_1(Y)$, $N_1(Y)$, and $H_2(Y;{\Bbb Z})$.

\begin{definition}
{\bf [generic length].} {\rm
 Let $C^{\prime}$ be an irreducible component of $C_{\varphi}$.
 For a closed point of $x^{\prime}\in C^{\prime}$,
 define
  $l(x^{\prime})
   = \length((_{{\cal O}_{C_{\varphi}}}{\cal E})|_{\widehat{x^{\prime}}})$
  ($= H^0((_{{\cal O}_{C_{\varphi}}}{\cal E})|_{\widehat{x^{\prime}}})$),
  where
   $\widehat{x^{\prime}}$ is the connected component of
   $\pi_{\varphi}^{-1}(\pi_{\varphi}(x^{\prime}))$
   that contains $x^{\prime}$.
 There exists an open dense subset $U^{\prime}\subset C^{\prime}$
  such that $l$ is constant on the set of all closed points
  in $U^{\prime}$.
 Define the {\it generic length} of
  $_{{\cal O}_{C_{\varphi}}}{\cal E}$ {\it on $C^{\prime}$}
  to be this constant.
}\end{definition}

\noindent
Cf.\ Figure~2-1-3.
\begin{figure}[htbp]
 \epsfig{figure=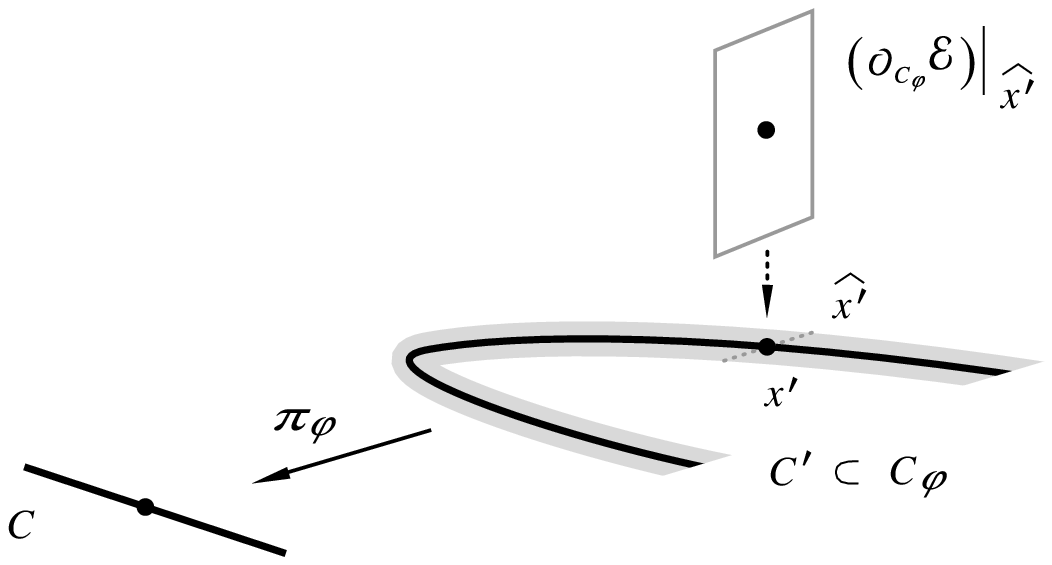,width=16cm}
 \centerline{\parbox{13cm}{\small\baselineskip 12pt
  {\sc Figure} 2-1-3.
  The length function $l(\,\bullet\,)$
   of $_{{\cal O}_{C_{\varphi}}}{\cal E}$ on $(C_{\varphi})_{\Bbb C}$.
  Here, a module over a scheme is indicated by $\dasharrow$.
  }}
\end{figure}

\begin{definition}
{\bf [image curve class].} {\rm
 Given a morphism $\varphi:(C^{\Azscriptsize},{\cal E})\rightarrow Y$,
 let
   $C_{\varphi}=\cup_iC^{\prime}_i$
    be the decomposition of $C_{\varphi}$ by irreducible components and
   $l_i$ be the generic length of $_{{\cal O}_{C_{\varphi}}}{\cal E}$
    on $C^{\prime}_i$.
 Then the {\it image curve class}, in notation $\varphi_{\ast}[C]$,
  in $A_1(Y)$ (similarly, in $N_1(Y)$ or $H_2(Y;{\Bbb Z})$)
  is defined by
   $$
    \varphi_{\ast}[C]\; :=\;
     \sum_i\,
       l_i\,\cdot\,
        (f_{\varphi})_{\ast}[(C^{\prime}_i)_{\redscriptsize}]\;
     \in\; A_1(Y)\,.
   $$
}\end{definition}

\begin{definition}
{\bf [combinatorial type].} {\rm
 The tuple from the above discussion, defined by
  $(g,r,\chi\,|\,\beta)
   := (g(C), \rank({\cal E}), \chi({\cal E})\,|\,\varphi_{\ast}[C])$,
  is called the {\it combinatorial type}
  of $\varphi:(C^{\Azscriptsize},{\cal E})\rightarrow Y$.
}\end{definition}

\begin{remark} {\rm
 (1)
 The notion of image curve class for a D-string world-sheet
  under a morphism has to go through the length function
  in Definition~2.1.11
  due to the complication that in general a coherent
   ${\cal O}_{X^{\prime}}$-module ${\cal E}^{\prime}$
   on a nonreduced scheme $X^{\prime}$ may not admit
   an open dense subscheme over which ${\cal E}^{\prime}$
   is locally free.
 Such situation can indeed happen
  for $_{{\cal O}_{C_{\varphi}}}{\cal E}$ on $C_{\varphi}$.

 (2)
 When $_{{\cal O}_{C_{\varphi}}}{\cal E}$
  is locally free on an open dense subset of $C_{\varphi}$,
  then the rank of such restrictions define the rank $r_i$
  of $_{{\cal O}_{C_{\varphi}}}{\cal E}$
  on each irreducible component $C^{\prime}_i$ of $C_{\varphi}$.
 In this case,
  the generic length $l_i$ of $_{{\cal O}_{C_{\varphi}}}{\cal E}$
   on an irreducible component $C^{\prime}_i$ of $C_{\varphi}$
   is given by $m_i\cdot r_i$,
    where $m_i$ is the multiplicity of $C^{\prime}_i$
    in the sense of [Fu],
   and
  $\varphi_{\ast}[C] = \sum_i\, m_ir_i\, \cdot\,
       (f_{\varphi})_{\ast}[(C^{\prime}_i)_{\redscriptsize}]$.
 This form fits into superstring-theory setting directly.

 (3)
 For morphisms to ${\Bbb P}^k$ such that Item (2) above applies,
  let $d_i$ be the degree of $f_{\varphi}|_{C^{\prime}_i}$.
 Then $\varphi_{\ast}[C]$ has degree $\sum_i r_id_i=rd$
  for some $d\in {\Bbb Z}_{\ge 0}$.
 Here, $r=\rank({\cal E})$.
 This illustrates that
  the notion of image curve class of D-string world-sheet in
   a target $Y$ involves not only the $1$-cycle
   that supports the D-string world-sheet
  but also the Chan-Paton module on the D-string world-sheet.
}\end{remark}

\bigskip

\begin{flushleft}
{\bf Morphisms over a (commutative) base scheme$/{\Bbb C}$.}
\end{flushleft}
The notion of a morphism from an Azumaya prestable curve to $Y$,
 as given in Definition~2.1.2,
 works for a general Azumaya scheme
 $X^{\Azscriptsize} := (X,{\cal O}_X^{\Azscriptsize})$,
  where
   $X=(X,{\cal O}_X)$ is a (commutative, Noetherian) scheme/${\Bbb C}$
    and
   ${\cal O}_X^{\Azscriptsize}
    =\Endsheaf({\cal E}):= \Endsheaf_{{\cal O}_X}({\cal E})$
    for a locally-free coherent ${\cal O}_X$-module ${\cal E}$ on $X$.
One can thus extend the notion of morphism to
 that of ``morphism over a base scheme" as follows:

\begin{definition} {\bf [morphism over a base and its pull-back].}
{\rm
 Let $S$ be a (commutative Noetherian) scheme over ${\Bbb C}$.
 A {\it family of morphisms over $S$} from Azumaya prestable curves
  with a fundamental module to $Y$,
  in notation
   $\Phi:=\Phi_S:({\cal C}^{\Azscriptsize}_S,{\cal E}_S)/S \rightarrow Y$,
  consists of the following data:
  \begin{itemize}
   \item[(1)]
    A flat family
     $({\cal C}^{\Azscriptsize}_S,{\cal E}_S)/S
      :=({\cal C}_S,{\cal O}_{{\cal C}_S}^{\Azscriptsize},
         {\cal E}_S)/S$,
      where ${\cal O}_{{\cal C}_S}^{\Azscriptsize}
             =\Endsheaf_{{\cal O}_{{\cal C}_S}}({\cal E}_S)$,
     of Azumaya prestable curves with a fundamental module..

   \item[(2)]
    An equivalence class of pairs
     $$
      ({\cal O}_{{\cal C}_S} \subset {\cal A}
                \subset {\cal O}_{{\cal C}_S}^{\Azscriptsize}\;,\;
        f:{\cal C}_{\cal A}:=\boldSpec{\cal A}\rightarrow Y)\,,
     $$
     where
     \begin{itemize}
      \item[(2.1)]
       ${\cal A}$ is a commutative ${\cal O}_{{\cal C}_S}$-subalgebra
        of ${\cal O}_{{\cal C}_S}^{\Azscriptsize}$;

      \item[(2.2)]
       $f:{\cal C}_{\cal A} \rightarrow Y$
       is a morphism of (commutative) schemes;

      \item[(2.3)]
       two such pairs
        $({\cal O}_{{\cal C}_S} \subset {\cal A}_1
                     \subset {\cal O}_{{\cal C}_S}^{\Azscriptsize}\;,\;
         f_1:{\cal C}_{{\cal A}_1}\rightarrow Y)$ and
        $({\cal O}_{{\cal C}_S} \subset {\cal A}_2
                     \subset {\cal O}_{{\cal C}_S}^{\Azscriptsize}\;,\;
         f_2:{\cal C}_{{\cal A}_2}\rightarrow Y)$
        are equivalent, in notation
        $$
         ({\cal O}_{{\cal C}_S} \subset {\cal A}_1
                     \subset {\cal O}_{{\cal C}_S}^{\Azscriptsize}\;,\;
          f_1:{\cal C}_{{\cal A}_1}\rightarrow Y)\;
         \sim\;
         ({\cal O}_{{\cal C}_S} \subset {\cal A}_2
                     \subset {\cal O}_{{\cal C}_S}^{\Azscriptsize}\;,\;
              f_2:{\cal C}_{{\cal A}_2}\rightarrow Y)\,,
        $$
       if there exists a third pair
        $({\cal O}_{{\cal C}_S} \subset {\cal A}_3
                     \subset {\cal O}_{{\cal C}_S}^{\Azscriptsize}\;,\;
         f_3:{\cal C}_{{\cal A}_3}\rightarrow Y)$
        such that
         ${\cal A}_3 \subset {\cal A}_i$ and that
         the induced diagram
          \begin{eqnarray*}
          \xymatrix{
           {\cal C}_{{\cal A}_i}\ar[drr]^{f_i}\ar[d] &&   \\
           {\cal C}_{{\cal A}_3}\ar[rr]^{f_3}        && Y \\
           }
          \end{eqnarray*}
          commutes, for $i=1,\, 2$.
     \end{itemize}
  \end{itemize}

 Let $h:T\rightarrow S$ be a morphism
  and ${\cal C}_T=h^{\ast}{\cal C}_S :=T\times_S{\cal C}_S$.
 Define the {\it pull-back} of $\Phi$ to $T$ to be the morphism
  $\Phi_T:= h^{\ast}\Phi_S:
   (h^{\ast}{\cal C}_S^{\Azscriptsize},h^{\ast}{\cal E}_S)/T
   \rightarrow Y$
  that corresponds to the equivalence class of pairs
   $({\cal O}_{{\cal C}_T} \subset {\cal A}^{\prime}
              \subset {\cal O}_{{\cal C}_T}^{\Azscriptsize}\;,\;
      f^{\prime}:{\cal C}_{{\cal A}^{\prime}}
                        :=\boldSpec{\cal A}^{\prime} \rightarrow Y)$
  that contains the pair
  $$
   ({\cal O}_{{\cal C}_T} \subset \overline{h^{\ast}{\cal A}_{\Phi}}
                        \subset {\cal O}_{{\cal C}_T}^{\Azscriptsize}\,,\,
    \overline{h^{\ast}f_{\Phi}}:
      \boldSpec\overline{h^{\ast}{\cal A}_{\Phi}} \rightarrow Y)\,,
  $$
  where
   \begin{itemize}
    \item[$\cdot$]
     ${\cal O}_{{\cal C}_S} \subset  {\cal A}_{\Phi}
          \subset {\cal O}_{{\cal C}_S}^{\Azscriptsize}$
      gives the surrogate ${\cal C}_{\Phi}$ of ${\cal C}_S^{\Azscriptsize}$
      associated to $\Phi$,
     it goes with a morphism $f_{\Phi}:{\cal C}_{\Phi}\rightarrow Y$,

    \item[$\cdot$]
     $\overline{h^{\ast}{\cal A}_{\Phi}}$ is the image
      ${\cal O}_{{\cal C}_T}$-subalgebra of
      the ${\cal O}_{{\cal C}_T}$-algebra homomorphism
      $h^{\ast}{\cal A}_{\Phi}\rightarrow
       {\cal O}_{{\cal C}_T}^{\Azscriptsize}
       = h^{\ast}{\cal O}_{{\cal C}_S}^{\Azscriptsize}$,
     it contains ${\cal O}_{{\cal C}_T}$ canonically
      as a ${\cal O}_{{\cal C}_T}$-subalgebra,

    \item[$\cdot$]
     $\overline{h^{\ast}f_{\Phi}}:
       \boldSpec\overline{h^{\ast}{\cal A}_{\Phi}} \rightarrow Y$
     is the composition of the canonical
     $\boldSpec\overline{h^{\ast}{\cal A}_{\Phi}}
      \rightarrow h^{\ast}{\cal C}_{\Phi}
       \stackrel{h^{\ast} f_{\Phi}}{\rightarrow} Y$.

   \end{itemize}
 When $\iota_s:\Spec{\Bbb C}\rightarrow S$ corresponds to
  a closed point $s\in S$,
 the pull-back $\Phi_s:=\iota_s^{\ast}\Phi$
  is called the {\it fiber} of $\Phi$ at $s\in S$.

 $\Phi:({\cal C}^{\Azscriptsize}_S,{\cal E}_S)/S \rightarrow Y$
  is said to be a {\it family of morphisms} over $S$
  {\it of combinatorial type} $(g,r,\chi\,|\,\beta)$
  if, in addition,
   $\Phi_s$ is of combinatorial type $(g,r,\chi\,|\,\beta)$
   for every closed point $s\in S$.
} \end{definition}

\begin{remark} {\it
 $[$surrogates in family and the significance of Chan-Paton modules$\,]$.}
{\rm
 Caution that,
  while both ${\cal O}_{{\cal C}_S}$ and
   ${\cal O}_{{\cal C}_S}^{\Azscriptsize}$ are flat over $S$,
 the ${\cal O}_{{\cal C}_S}$-module ${\cal A}_{\Phi}$ on ${\cal C}_S$
  that gives the family ${\cal C}_{\Phi}$ of surrogates associated
  to $\Phi$ in general is {\it not} flat over $S$.
 {\it Nor} does the surrogate pass from one to another directly
  via the usual pull-back under base change.
 This is an indication that the pure moduli problem of morphisms
  from Azumaya prestable curves to $Y$ is a technically worse problem
 than the moduli problem of D-strings.
 In the latter, the extra data of Chan-Paton modules on D-branes
  enlarges the moduli problem while taming it significantly.
 This will become clearer from the point of view of
  the next subsection.
} \end{remark}

%
%
%
%
%
%
%
%
%
%
%

\bigskip

\subsection{Azumaya without Azumaya, morphisms without morphisms.}

\begin{flushleft}
{\bf Recasting into commutative algebro-geometric setting.}
\end{flushleft}
A morphism from an Azumaya prestable curve with a fundamental module
 to $Y$ under our setting can be redescribed/recast
 effectively and faithfully in terms of commutative geometric data.
Two equivalent such data are given below.

\begin{lemma}
{\bf [Azumaya without Azumaya, morphisms without morphisms].}
 Given a prestable curve $C$,
  a morphism $\varphi:(C^{\Azscriptsize},{\cal E})\rightarrow Y$
   from an Azumaya curve $C^{\Azscriptsize}$ over $C$
    with a fundamental module ${\cal E}$ of rank $r$ is given by
  a coherent ${\cal O}_{C\times Y}$-module $\widetilde{\cal E}$
   on $(C\times Y)/C$ of relative length $r$.
 The correspondence is functorial and bijective.
\end{lemma}

\begin{proof}
 Given $\varphi$,
  it follows from Sec.~2.1 that
  one can canonically associated to it
   an ${\cal O}_{\Gamma_{\varphi}}$-module
   $\widetilde{\cal E} := \; _{{\cal O}_{C_{\varphi}}}\!{\cal E}$
   via the canonical isomorphism
    $C_{\varphi}\simeq \Gamma_{\varphi}\subset C\times Y$.
 As an ${\cal O}_{C\times Y}$-module on $(C\times Y)/C$,
  $\widetilde{\cal E}$ satisfies the stated properties.

 Conversely, given $\widetilde{\cal E}$ on $(C\times Y)/C$ as stated,
  let
   $\Gamma=\Supp{\widetilde{\cal E}}$,
   ${\cal E}:= \pr_{1\ast} \widetilde{\cal E}$, and
   ${\cal O}_C^{\Azscriptsize} :=\Endsheaf{\cal E}$.
  Then the tautological action of ${\cal O}_{\Gamma}$
   on $\widetilde{\cal E}$ gives rise to
   an action of ${\cal O}_{\Gamma}$ on ${\cal E}$.
  This realizes ${\cal O}_{\Gamma}$ canonically as
   a commutative ${\cal O}_C$-sub-algebra ${\cal A}$
    in ${\cal O}_C^{\Azscriptsize}$.
  The morphism $\Gamma\rightarrow Y$ from the restriction of
   $\pr_2:C\times Y\rightarrow Y$ gives rise to
   a morphism $f:C_{\cal A}:=\boldSpec{\cal A}\rightarrow Y$
   that satisfies the minimal property in
   Remark~2.1.4.
 Thus, $\widetilde{\cal E}$ corresponds canonically to
  $\varphi:(C^{\Azscriptsize},{\cal E})\rightarrow Y$
  given by ${\cal E}$ and the pair
  $( {\cal O}_C \subset {\cal A}
        \subset {\cal O}_C^{\Azscriptsize}:=\Endsheaf{\cal E}\,,\,
     f:C_{\cal A}\rightarrow Y )$ as defined.

 The functoriality and bijectivity of the correspondence follow
  from the fact that the above two-way correspondence is canonical.

\end{proof}

\noindent
Cf.\ Figure~2-2-1.$\;$
It is worth emphasizing that
 $\widetilde{\cal E}$ on $C\times Y$ in the above lemma
 is flat over $C$ and, hence,
its push-forward $\pr_{1\ast}\widetilde{\cal E}$ to $C$
 is locally free.

\begin{figure}[htbp]
 \epsfig{figure=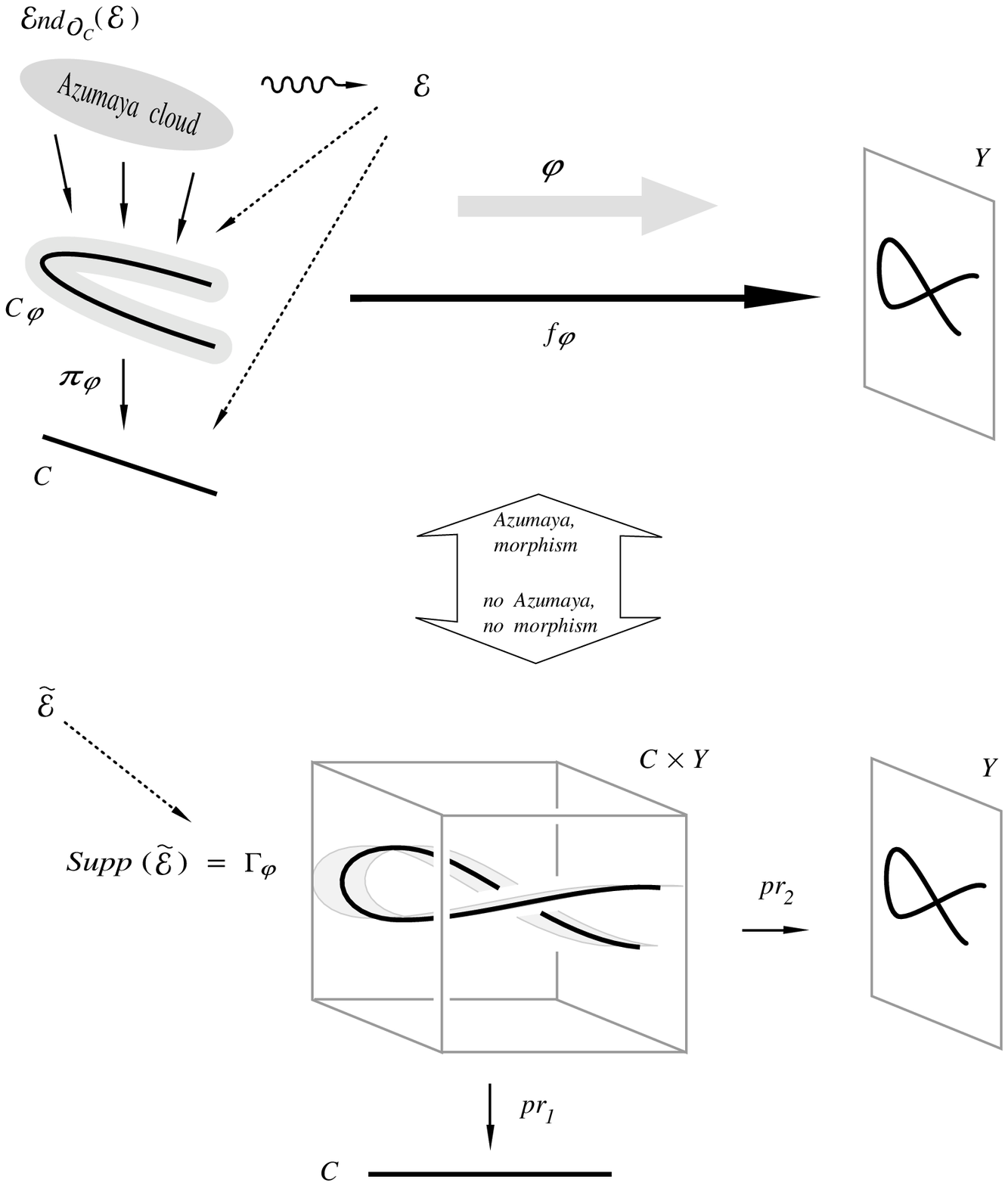,width=16cm}
 \centerline{\parbox{13cm}{\small\baselineskip 12pt
  {\sc Figure} 2-2-1.
   The correspondence between
    $\{\,$morphisms $\varphi:(C^{A\!z},{\cal E})\rightarrow Y$$\,\}$
     and
    $\{\,$torsion ${\cal O}_{C\times Y}$-modules $\widetilde{\cal E}$
     on $C\times Y$ that are flat over $C$
     with relative dimension $0$$\,\}$.
   Here, only $Y$ is fixed and
    the correspondence is at the level of related moduli stacks.
   In the figure, a module over a scheme is indicated by a dash-arrow
    $\dasharrow$.
  }}
\end{figure}

\begin{remark}
 $[$AwoA $\&$ MwoM in general$\,]$.
{\rm
 Replacing the prestable curve $C$ in Lemma~2.2.1
  by a scheme $X$ gives a canonical description of
  a morphism $\varphi:(X^{\!\Azscriptsize},{\cal E})\rightarrow Y$
  in terms of a coherent ${\cal O}_{X\times Y}$-module
  $\widetilde{\cal E}$ of relative length $r$ on $(X\times Y)/X$.
}\end{remark}

Lemma~2.2.1 and Remark~2.2.2 allow us to bring the language and
 techniques in purely commutative algebraic geometry into our study.
In this way we are tossed back to the category of commutative geometry.

\bigskip

Applying Lemma~2.2.1 and Remark~2.2.2 to the case of D$0$-branes
 (cf. [Va] and [L-Y2: Sec.~4]),
one then realizes the moduli stack of D$0$-branes of type $r$ on $Y$
 as the algebraic stack of $0$-dimensional, length $r$,
 torsion sheaves on $Y$;
see Sec.~3.1 for more discussions.
It follows immediately that

\begin{corollary}
{\bf [topological D-string as a smearing of D0-branes].}
 A morphism $\varphi:(C^{\Azscriptsize},{\cal E})\rightarrow Y$
  corresponds canonically to a morphism from $C$ to the stack of
  D$0$-branes on $Y$; and vice versa.
\end{corollary}

Replacing $C$ in the corollary by a scheme $X$
 realizes Euclidean $D$-branes of B-type as
 morphisms from $X$ to the stack of D$0$-branes on $Y$.
This gives a precise realization of general D$p$-branes of B-type
 as ``smearing" of D$0$-branes along a scheme/cycle,
cf.\ [L-Y2: the paragraph before Remark 2.2.5] and Figure~3-1-1.

\bigskip

Definition~2.1.12
 of image curve class $\beta$ of a D-string world-sheet
 under a morphism follows the feature of D-branes in superstring theory.
The following lemma shows that it is indeed a ``correct" definition
 for the current prototype mathematical formulation of D-strings
 in Definition~2.1.6 and Remark~2.1.7
 along the Polchinski-Grothendieck Ansatz:

\begin{lemma}
{\bf [invariance of combinatorial type under flat deformation].}
 Let
  ${\cal C}_S/S$ be a flat family of prestable curves
   over (a connected) $S$,
  $\widetilde{\cal E}_S$ be an ${\cal O}_{{\cal C}_S\times Y}$-module
   of relative length $r$ on $({\cal C}_S\times Y)/{\cal C}_S$,
   (in particular, $\widetilde{\cal E}_S$ on ${\cal C}_S\times Y$
    is flat over $S$), and
  $\varphi_s:(C_s^{\Azscriptsize},{\cal E}_s)\rightarrow Y$
   be the morphism, with a fundamental module,
   associated to $\widetilde{\cal E}_s$
   for a closed point $s\in S$.
 Then, the combinatorial type $(g,r,\chi\,|\,\beta)$
  of $\varphi_s$ is constant over $S$.
\end{lemma}

\begin{proof}
 Let ${\cal O}_Y(1)$ be the fixed polarization on $Y$ and
  take ${\cal O}_{C\times Y}(1)$ to be
   ${\cal O}_C(1)\boxtimes {\cal O}_Y(1)$.
 Let $\widetilde{\cal E}$ be an ${\cal O}_{C\times Y}$-module of
  relative length $r$ on $(C\times Y)/C$.
 Then, $\Supp\widetilde{\cal E}$ is $1$-dimensional and,
  hence, there exists an $\widetilde{\cal E}$-regular hyperplane
   $H= \pr_1^{\ast}H_C+\pr_2^{\ast}H_Y \in |{\cal O}_{C\times Y}(1)|$
    on $C\times Y$
    with $H_C$ (resp.\ $H_Y$) being a hyperplane in $C$ (resp. $Y$)
     with respect to ${\cal O}_C(1)$ (resp.\ ${\cal O}_Y(1)$)
   such that
    $(\pr_1^{\ast}H_C\cap \pr_2^{\ast}H_Y)
                    \cap \Supp(\widetilde{\cal E})$
    is empty.
 Since the morphism $\Supp(\widetilde{\cal E})\rightarrow C$
  from the restriction
   of the projection map $\pr_1:C\times Y\rightarrow C$
  is affine, one has the values of Hilbert polynomials
  $P(\widetilde{\cal E},\,\bullet\,)$ of $\widetilde{\cal E}$
  at $0$ and $1$ given respectively by
  $$
   P(\widetilde{\cal E},0)\;=\; \chi(\widetilde{\cal E})\;
    =\; \chi(\pr_{1\ast}\widetilde{\cal E})\;
    =\; \chi
  $$
  and
  $$
   P(\widetilde{\cal E},1)\;=\; \chi(\widetilde{\cal E}(1))\;
    =\; \chi\left(\widetilde{\cal E}|_H\right) {1\choose 1} +\chi_0\;
    =\; (r\deg(C) + H_Y\cdot\beta)+\chi\,.
  $$
 This implies that
  $$
   P(\widetilde{\cal E},m)\;
    =\; (r\deg(C) + H_Y\cdot\beta)m + \chi\,.
  $$
 In particular, it is completely determined
  by the combinatorial type $(g,r,\chi\,|\,\beta)$
  of the associated
  $\varphi:(C^{\Azscriptsize},{\cal E})\rightarrow Y$ and
  the polarization chosen on $C$ and $Y$ respectively.
 The lemma follows.

\end{proof}

\begin{remark}
{\rm [{\it Euler characteristic vs.\ degree}].}
{\rm
 Note that
  for a locally free sheaf ${\cal E}$ of rank $r$
   on a prestable curve $C$ of (arithmetic) genus $g$
    with a fixed polarization ${\cal O}_C(1)$,
  its Hilbert polynomial $P({\cal E},m)$ is given by
   $$
    P({\cal E},m)\;=\; r\deg(C)m + \deg({\cal E}) + r(1-g)\,.
   $$
 This follows from
  the Riemann-Roch Theorem for smooth curves,
  additivity of Hilbert polynomials with respect to an short
    exact sequence, and
  the consideration of the canonical short exact sequence of
   ${\cal O}_C$-modules
   $0\rightarrow {\cal E}\rightarrow \nu_{\ast}\nu^{\ast}{\cal E}
     \rightarrow {\cal T}\rightarrow 0$,
   where $\nu:C^{\prime}\rightarrow C$ is the normalization of $C$
    and note that ${\cal T}$ is a torsion sheaf supported on
    the set of nodes of $C$ with the fiber dimension at each node
    equal to $r$.
 Thus, in our problem, $\chi$ in $(g,r,\chi\,|\,\beta)$
  can be replaced by the degree of fundamental modules.
}\end{remark}

\bigskip

\begin{flushleft}
{\bf Isomorphisms.}
\end{flushleft}
\begin{definition}
{\bf [isomorphism].} {\rm
 An ${\cal O}_{C_1\times Y}$-module $\widetilde{\cal E}_1$
  of relative length $r$ on $(C_1\times Y)/C_1$ and
 an ${\cal O}_{C_2\times Y}$-module $\widetilde{\cal E}_2$
  of relative length $r$ on $(C_2\times Y)/C_2$
 are said to be {\it isomorphic}
 if there exists a pair $(h,\widetilde{h})$,
 where
  $h:C_1\simeq C_2$ and
  $\widetilde{h}:\widetilde{\cal E}_1
               \simeq h^{\ast}\widetilde{\cal E}_2$
  are isomorphisms of schemes and coherent sheaves respectively.
Here, we denote the $h$-induced isomorphisms
 $C_1\times Y\simeq C_2\times Y$ between the products still by $h$.
}\end{definition}

The following lemma follows from
 the correspondence between the noncommutative and the commutative
  picture of morphisms, given in the proof of Lemma~2.2.1,
 and a definition-tracing:

\begin{lemma}
{\bf [isomorphism: noncommutative vs.\ commutative picture].}
 An ${\cal O}_{C_1\times Y}$-module $\widetilde{\cal E}_1$
   of relative length $r$ on $(C_1\times Y)/C_1$ and
  an ${\cal O}_{C_2\times Y}$-module $\widetilde{\cal E}_2$
   of relative length $r$ on $(C_2\times Y)/C_2$
  are isomorphic
 if and only if
  the corresponding morphisms
    $\varphi_1:(C_1,{\cal E}_1)\rightarrow Y$ and
    $\varphi_2:(C_2,{\cal E}_2)\rightarrow Y$
   are isomorphic in the sense of Definition~2.1.5.
\end{lemma}

It follows from
 the immediate generalization of Lemma~2.2.1 to families and
 Lemma~2.2.7 that:

\begin{proposition}
{\bf [stack of morphisms as stack of sheaves on products].}
 The stack of morphisms from Azumaya prestable curves
   with a fundamental module to $Y$
  is canonically isomorphic to the stack of torsion sheaves of modules
   on the product of prestable curves with $Y$
   of finite relative length over the prestable curves.
\end{proposition}

We will denote the stack of morphisms
 from Azumaya prestable curves with a fundamental module
 to $Y$ of combinatorial type $(g,r,\chi\,|\,\beta)$
 by ${\frak M}_{\Azscriptsize(g,r,\chi)^f}(Y,\beta)$.
This is an algebraic stack over ${\Bbb C}$,
 but only locally of finite-type.
We will describe an atlas of
 ${\frak M}_{\Azscriptsize(g,r,\chi)^f}(Y,\beta)$
 in Sec.~3.2.

\bigskip

\subsection{Boundedness of morphisms.}

\begin{proposition}
{\bf [boundedness of morphisms].}
 Let $({\cal C}^{\Azscriptsize}_S, {\cal E}_S)/S$
  be a family of Azumaya prestable curves with a fundamental module
  of type $(g,r,\chi)$ over a base scheme $S$ (of finite type).
 Then the set $\{\varphi_{\bullet}\}_{\bullet}$ of morphisms
  from fibers of $({\cal C}^{\Azscriptsize}_S,{\cal E}_S)/S$ to $Y$
  of type $(g,r,\chi\,|\,\beta)$ is bounded.
\end{proposition}

\begin{proof}
 Fix
  a relative ample line bundle ${\cal O}_{{\cal C}^{A\!z}_S/S}(1)$
    on ${\cal C}^{\Azscriptsize}_S/S$  and
  an ample line bundle ${\cal O}_Y(1)$ on $Y$
    with the associated hyperplane class $H_Y$.
 Let
  ${\cal O}_{{\cal C}^{A\!z}_S\times Y}(1)
   :=  {\cal O}_{{\cal C}^{A\!z}_S/S}(1) \boxtimes {\cal O}_Y(1)$.
 A morphism $\varphi$ from $(C^{\Azscriptsize}_s,{\cal E}_s)$ to $Y$,
   where $s$ is a closed point of $S$,
  of type of $(g,r,\chi\,|\,\beta)$
  is given by an ${\cal O}_{{\cal C}^{A\!z}_s\times Y}$-module
  $\widetilde{\cal E}_s$ such that
   $\widetilde{\cal E}_s$
    is of relative length $r$ over $C^{\Azscriptsize}_s$ and
    has the Hilbert polynomial
     $P(\widetilde{\cal E}_s,m) = (r\deg(C) + H_Y\cdot\beta)m + \chi$
   and that $\pr_{1\ast}(\widetilde{\cal E}_s)\simeq {\cal E}_s$.
 As $\Supp(\widetilde{\cal E}_s)$ is affine over $C_s$,
 once fixing the isomorphism
   $\pr_{1\ast}(\widetilde{\cal E}_s)\simeq {\cal E}_s$,
  the tautological morphism
   $\pr_1^{\ast}({\cal E}_s)
     \rightarrow \widetilde{\cal E}_s \rightarrow 0$
   of ${\cal O}_{C_s\times Y}$-modules
   is surjective.
 Since ${\cal E}_S/S$ is given, it follows that the set of
  all such $\widetilde{\cal E}_s$ on $C_s\times Y$ is bounded,
  cf.\ [Gr] and [H-L: Lemma 1.7.6].
 The proposition now follows from Lemma~2.2.1.

\end{proof}

Continuing the notation in the above proof,
 as the set of $\widetilde{\cal E}_s$ is bounded,
 so is the set of their (scheme-theoretic) support
 $\Supp(\widetilde{\cal E}_s)$,
 described by the ideal sheaf
  $\Ker({\cal O}_{C_s\times Y}
         \rightarrow \Endsheaf(\widetilde{\cal E}_s))$.
A bound of $\chi(\Supp(\widetilde{\cal E}_s))$,
 hence a bound of the (arithmetic) genus $1-\chi$ of the surrogate
 associated to morphisms of type $(g,r,\chi\,|\,\beta)$
 from fibers of $({\cal C}^{\Azscriptsize}_S,{\cal E}_S)/S$ to $Y$,
can be expressed in terms of Hilbert polynomials and
 the Mumford-Castelnuovo regularity of certain sheaves involved
 as follows.
Let
 $$
  m_0\; :=\;
   \max \left\{ \reg({\cal F})\,:\,
          \begin{array}{l}
           \mbox{${\cal F}\:=\:
                  {\cal O}_{C_s\times Y}\,$  or
                  $\,\Ker({\cal O}_{C_s\times Y}
                    \rightarrow \Endsheaf(\widetilde{\cal E}_s))$}
                                                           \\[.6ex]
           \mbox{from the above bounded family}
          \end{array}
        \right\}\,.
 $$
Here, $\reg({\cal F})$ is the Mumford-Castelnuovo regularity
 for an ${\cal O}_{C_s\times Y}$-module ${\cal F}$
 with respect to ${\cal O}_{C_s\times Y}(1)$.
Since $\reg(\,\bullet\,)$ here is taken over a bounded family of
 modules, $m_0$ is finite in ${\Bbb Z}_{\ge 0}$.
It follows that the Hilbert polynomials satisfy inequalities
 $$
  0\; \le\;  P(\Supp(\widetilde{\cal E}_s), m)\;
      \le \; P({\cal O}_{C_s\times Y}, m)
  \hspace{2ex}\mbox{for all $\,m\ge m_0$}\,.
 $$
Note that $P({\cal O}_{C_s\times Y}, m)$
 is a polynomial $P(m)$ in $m$, independent of $s\in S$.
It follows that
 $$
  P(\Supp(\widetilde{\cal E}_s), m)\;
   =\; \alpha(\Supp(\widetilde{\cal E}_s))\, m\,
       +\, \chi(\Supp(\widetilde{\cal E}_s))
 $$
 with
 $$
  1\; \le\; \alpha(\Supp(\widetilde{\cal E}_s))\;
      \le\; r\deg(C)+H_Y\cdot\beta
 $$
 and
 $$
  -(r\deg(C)+H_Y\cdot\beta)\,m\;
   \le\; \chi(\Supp(\widetilde{\cal E}_s))\;
   \le\; P(m) - m  \hspace{1em}\mbox{for all $\,m\ge m_0$}\,.
 $$
In particular,
 $$
  1-P(m_0)+m_0\;
   \le\; g(C_{\varphi})\;
   \le\; 1\,+\, (r\deg(C)+H_Y\cdot\beta)\,m_0\,.
 $$

\bigskip

\subsection{Morphisms from an Azumaya prestable curve to ${\Bbb P}^k$.}

Morphisms to a projective space ${\Bbb P}^k$ deserves
 some attention.
In the commutative case, there are three presentations of
 a morphism from a scheme $X$ to a projective space:
 (1) as a morphism between schemes,
 (2) as a quotient ${\cal O}_X^{k+1}\rightarrow {\cal L}\rightarrow 0$
      to a line bundle, and
 (3) as a $(k+1)$-tuple of rational functions on $X$.
Let $Y={\Bbb P}^k$, then Sec.~2.1 gives the presentation for
 morphisms from Azumaya prestable curves to ${\Bbb P}^k$
 from Aspect (1).
In the current subsection, we discuss the presentation
 of such morphisms from Aspect (3).\footnote{It
                         would be nice to have a presentation of
                          morphisms from Azumaya curves to ${\Bbb P}^k$
                          in terms of Aspect (2) as well.
                         In some sense, a fundamental module ${\cal E}$
                          of $C^{A\!z}$ may be thought of as a line bundle
                          on ${\cal C}^{\Azscriptsize}$.}

We formulate first a format of localizations that is akin to
 the problem and then define and study morphisms from Azumaya curves
 to ${\Bbb P}^k$.
This subsection is a consequence of [L-Y2].
Details that follow ibidem immediately are omitted.

\bigskip

\begin{flushleft}
{\bf Localizations and covers of commutative ${\cal O}_C$-subalgebras
      of ${\cal O}_C^{\Azscriptsize}$.}
\end{flushleft}
Let ${\cal O}_C\subset {\cal A}\subset {\cal O}_C^{\Azscriptsize}$ be
 a commutative ${\cal O}_C$-subalgebra of ${\cal O}_C^{\Azscriptsize}$.
Let ${\cal K}_C$ (resp.\ ${\cal K}_{C_{\cal A}}$)
 be the sheaf of total quotient rings of ${\cal O}_C$
 (resp.\ ${\cal O}_{C_{\cal A}}$), cf.\ [Ha].
An element $h$ of $\Gamma({\cal K}_{C_{\cal A}})$
 is called a {\it rational function}\footnote{We
                               do not distinguish two rational functions
                                $h_1$ and $h_2$ on a scheme $X$
                               whose domains of definition
                                $V_{h_1}\subset X$ and
                                $V_{h_2}\subset X$
                                share a common open dense subset
                                $V\subset V_{h_1}$ and
                                $V\subset V_{h_2}$.
                               Similarly for rational sections of
                                an ${\cal O}_X$-module.          }
 on $C_{\cal A}$.
As the tautological morphism $\pi^{\cal A}:C_{\cal A}\rightarrow C$
 is finite,
${\cal K}_{C_{\cal A}}\simeq {\cal A}\otimes_{{\cal O}_C}{\cal K}_C$
 canonically as ${\cal A}$-algebras.

\begin{notation} {\rm
 The element in $\Gamma({\cal A}\otimes_{{\cal O}_C}{\cal K}_C)$
  associated to $h\in\Gamma({\cal K}_{C_{\cal A}})$
  is denoted by $m_h$ and
 the element in $\Gamma({\cal K}_{C_{\cal A}})$ associated to
  $m\in\Gamma({\cal A}\otimes_{{\cal O}_C}{\cal K}_C)$
  is denoted by $h_m$.
} \end{notation}

\noindent
The notion of
  function rings of local charts and their localizations and
  that of open covers of $C_{\cal A}$
  in commutative algebraic geometry
 can be re-phrased in terms of localizations of
 ${\cal O}_C$-subalgebras of ${\cal A}\otimes_{{\cal O}_C}{\cal K}_C$,
 as follows.

Recall that, for two commuting idempotents $e$ and $e^{\prime}$
 in the matrix algebra $M_r({\Bbb C})$ over ${\Bbb C}$,
we say that
 $e^{\prime}$ is {\it subordinate to} $e$
 if there exists an idempotent $e^{\prime\prime}\in M_r({\Bbb C})$
  that commutes with both $e$ and $e^{\prime}$
   such that $e=e^{\prime}+e^{\prime\prime}$.
In notation, $e^{\prime}\preceq e$ or $e\succeq e^{\prime}$.

Let $C=\cup_iC_i$ be the decomposition of $C$ into irreducible
 components.
A {\it pseudo-section} (resp.\ {\it local pseudo-section}) $s$ of
 an ${\cal O}_C$-module ${\cal F}$ is by definition a tuple $(s_i)_i$,
 where $s_i$ is a section (resp.\ local section) of ${\cal F}|_{C_i}$.

\begin{definition} {\bf [admissible idempotent pseudo-section].} {\rm
 An {\it admissible idempotent pseudo-section} of
  ${\cal O}_C^{\Azscriptsize}$ is a tuple $e:=(e_i)_i$,
   where
   \begin{itemize}
    \item[(1)]
     $e_i$ is an idempotent section of
      ${\cal O}_{C_i}^{\Azscriptsize}
        := \left.{\cal O}_C^{\Azscriptsize}\right|_{C_i}$
      of constant rank as an endomorphism of ${\cal E}|_{C_i}$;

    \item[(2)]
     for each node $p\in C_i\cap C_j$,
      $e_i(p)$ and $e_j(p)$ commute and
      either $e_i(p)\preceq e_j(p)$ or $e_i(p)\succeq e_j(p)$.
   \end{itemize}
 The {\it addition} $e_1+e_2$ and {\it multiplication} $e_1e_2$
  of two commuting idempotent pseudo-sections
  $e_1 =(e_{1,i})_i$ and $e_2=(e_{2,i})_i$
  are defined by $(e_{1,i}+e_{2,i})_i$ and $(e_{1,i}e_{2,i})_i$
  respectively.
 Similarly, the multiplications $me$ and $em$ of
  a local section $m$ of ${\cal O}_C^{\Azscriptsize}$ with $e$
  are defined component by component.
 Both $me$ and $em$ are local pseudo-sections of
  ${\cal O}_C^{\Azscriptsize}$.
 The same notions are defined for any ${\cal O}_C$-subalgebra
  ${\cal O}_C\subset {\cal A}\subset {\cal O}_C^{\Azscriptsize}$.
} \end{definition}

\begin{definition} {\bf [sheaf of algebra-subsets].} {\rm
(Cf.\ [L-Y2: Definition 3.2.1].)
 (Continuing the notations from above.)
 Let $e=(e_i)_i$ be an admissible idempotent pseudo-section
  of ${\cal O}_C^{\Azscriptsize}$.
 Then  $e$ specifies a sheaf ${{\cal O}_C^{\Azscriptsize}}_{(e)}$ of
  ${\cal O}_C$-submodules of ${\cal O}_C^{\Azscriptsize}$
  that consists of local sections $m$ of ${\cal O}_C^{\Azscriptsize}$
  that satisfy $me=em=m$.
 The restriction
  $\left.({{\cal O}_C^{\Azscriptsize}}_{(e)}\,,\,e)\right|_{C_i}$
  is a sheaf of ${\cal O}_{C_i}$-algebra-subsets of
  $\left.{\cal O}_C^{\Azscriptsize}\right|_{C_i}$
  with the section of identities given by $e_i$.
 With a slight abuse of terminology,
  we shall call ${{\cal O}_C^{\Azscriptsize}}_{(e)}$
  the {\it sheaf of ${\cal O}_C$-algebra-subsets of
           ${\cal O}_C^{\Azscriptsize}$ determined by $e$}.
} \end{definition}

\noindent
Cf.\ Figure~2-4-1.
\begin{figure}[htbp]
 \epsfig{figure=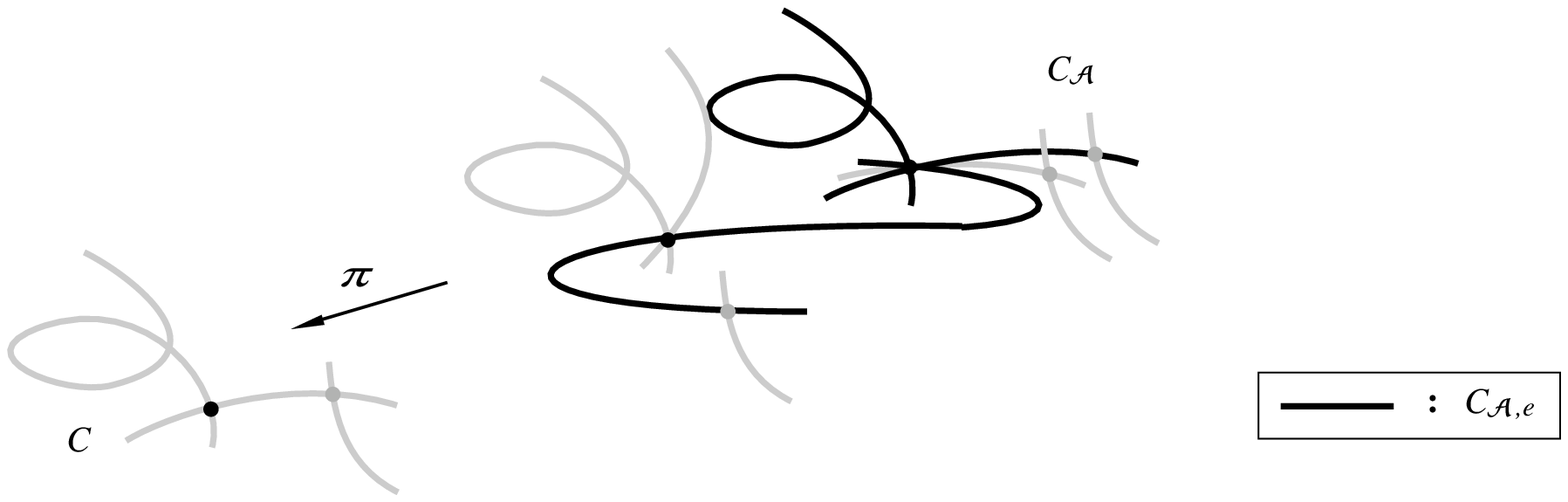,width=16cm}
 \centerline{\parbox{13cm}{\small\baselineskip 12pt
  {\sc Figure} 2-4-1.
  The meaning of $e$:
  $e$ selects for each surrogate $C_{\cal A}$ of $C^{A\!z}$
   a subcurve $C_{{\cal A},e}$
   that is maximal with respect to the property that
   $\smallboldSpec({\cal A}\cap {\cal O}^{A\!z}_{C(e)})$
   is dominated by it as ${\Bbb C}$-schemes.
  ($C_{{\cal A},e}$ can be empty.)
  }}
\end{figure}

\begin{example}
{\bf $[$${{\cal O}_C^{\Azscriptsize}}_{(e)}$$]$.} {\rm
 If $e_i=0$ for some $C_i$,
  then $\left.{{\cal O}_C^{\Azscriptsize}}_{(e)}\right|_{C_i}$
   is the zero ${\cal O}_{C_i}$-algebra-subsets
   of $\left.{\cal O}_C^{\Azscriptsize}\right|_{C_i}$.
 For $e=(e_i)_i$ associated to $1\in \Gamma({\cal O}_C^{\Azscriptsize})$,
  ${{\cal O}_C^{\Azscriptsize}}_{(e)}={\cal O}_C^{\Azscriptsize}$.
} \end{example}

\begin{lemma}
{\bf [characteristic function vs.\ admissible idempotent pseudo-section].}
 Let
  ${\cal O}_C\subset {\cal A}\subset {\cal O}_C^{\Azscriptsize}$ be
   a commutative ${\cal O}_C$-subalgebra of ${\cal O}_C^{\Azscriptsize}$
    and
  $C_{\cal A}:=\boldSpec{\cal A}=C_{{\cal A},1}\cup C_{{\cal A},2}$
   be a decomposition of $C_{\cal A}$ by two disjoint sets
    of irreducible components.
 Let $\chi$ be the characteristic pseudo-function on $C_{\cal A}$
  with support $C_{{\cal A},1}$
  (i..e\ it takes value $1$ on $C_{{\cal A},1}$ and $0$ on
         $C_{\cal A}-C_{{\cal A},1}$).
 Then
  $\chi$ corresponds to a unique admissible idempotent pseudo-section
   $e_{\chi}$ of ${\cal A}$
  such that
   for any rational function $h$ on $C_{\cal A}$ with
    $h|_{C_{{\cal A},2}}=0$,
   $m_he_{\chi}\, (=e_{\chi}m_h)\,=m_h$.
\end{lemma}

Let
 ${\cal O}_C\subset {\cal A}\subset {\cal O}_C^{\Azscriptsize}$ be
  a commutative ${\cal O}_C$-subalgebra of ${\cal O}_C^{\Azscriptsize}$
   and
 $h$ be a rational function on $C_{\cal A}:=\boldSpec{\cal A}$.
Let $\chi_h$ be the characteristic pseudo-function on $C_{\cal A}$
 with support the union of the irreducible components of $C_{\cal A}$
  to which the restriction of $h$ is not nilpotent.

\begin{lemma}
{\bf [inverse of rational section and open set].}
 Let
  $e_{\chi_h}$ be the admissible idempotent pseudo-section
   of ${\cal A}$ associated to $\chi_h$  and
  ${\cal A}_{(e_{\chi_h})}
   := {\cal A}\cap {{\cal O}_C^{\Azscriptsize}}_{(e_{\chi_h})}$,
   which is an ${\cal O}_C$-algebra-subset of ${\cal A}$
    with the section of identities $e_{\chi_h}$.
 Then, $m_he_{\chi_h}$ is invertible in
  $\Gamma({\cal A}_{(e_{\chi_h})}\otimes_{{\cal O}_C}{\cal K}_C)$.
 It follows that the open subcurve of $C_{\cal A}$ defined as
   the regular locus of $1/h$ in the support of $\chi_h$
  corresponds to the localization
   ${\cal A}_{e_{\chi_h}}[(m_he_{\chi_h})^{-1}]$
   of the ${\cal O}_C$-algebra-subset ${\cal A}_{e_{\chi_h}}$
   in ${\cal A}_{(e_{\chi_h})}\otimes_{{\cal O}_C}{\cal K}_C$.
\end{lemma}

Given now a commutative ${\cal O}_C$-subalgebra
 $({\cal O}_C\subset)\,{\cal A}\subset {\cal O}_C^{\Azscriptsize}$,
a collection $\{\, ({\cal A}_{\alpha}\,,\,e_{\alpha})\,\}_{\alpha\in I}$
 of sheaves of ${\cal O}_C$-algebra-subsets of
  ${\cal A}\otimes_{{\cal O}_C}{\cal K}_C$,
  where $e_{\alpha}$ is an admissible idempotent pseudo-section
   of ${\cal A}$ that serves as the identity section of
   ${\cal A}_{\alpha}$ for $\alpha\in I$,
 that satisfy the following conditions:
  \begin{itemize}
   \item[$(1)$]
    fix a well-ordering of the index set $I$, then
    \begin{eqnarray*}
     \lefteqn{
      1\;=\; \sum_{\alpha} e_{\alpha}\;
            -\, \sum_{\alpha_1<\alpha_2} e_{\alpha_1}e_{\alpha_2}\;
            +\, \sum_{\alpha_1<\alpha_2<\alpha_3}
                   e_{\alpha_1}e_{\alpha_2}e_{\alpha_3} }\\[.6ex]
      && \hspace{6em}
          \pm\; \cdots\;
            +\, (-1)^{|I|+1}\,
                 \sum_{\alpha_1<\,\cdots\,<\alpha_{|I|}}
                     e_{\alpha_1}\,\cdots\,e_{\alpha_{|I|}}\,;
    \end{eqnarray*}

   \item[$(2)$]
    $( {\cal O}_C
         \langle {\cal A}_{\alpha_1}
                      \cdots\, {\cal A}_{\alpha_l}\rangle\,,\,
         e_{\alpha_1}\cdots\,e_{\alpha_l} )$
    is a sheaf of ${\cal O}_C$-algebra-subsets
     in ${\cal A}\otimes_{{\cal O}_C}{\cal K}_C$,\\
     for all $\{\alpha_1\,,\cdots\,,\alpha_l\}\subset I$;

   \item[$(3)$]
    for $I_0\subset I$, let
    \begin{eqnarray*}
     e_{I_0}  & = & \prod_{\alpha\in I_0}e_{\alpha}\,,\\[1.2ex]
     \hat{e}_{I_0}
      & = & 1 - \sum_{\alpha^{\prime}} e_{(\alpha^{\prime})}
              + \sum_{\alpha^{\prime}_1<\alpha^{\prime}_2}
                   e_{(\alpha^{\prime}_1)}e_{(\alpha^{\prime}_2)}
                + \,\cdots\,\\[.6ex]
      && \mbox{\hspace{4em}}
         + (-1)^l\sum_{\alpha^{\prime}_1<\,\cdots\,<\alpha^{\prime}_l}
               e_{(\alpha^{\prime}_1)}\,\cdots\,e_{(\alpha^{\prime}_l)}
         +\,\cdots\,
         + (-1)^{|I|-|I_0|}
               e_{(\alpha^{\prime}_1)}\,
                           \cdots\,e_{(\alpha^{\prime}_{|I|-|I_0|})}\,,
    \end{eqnarray*}
    where $|I|$, $|I_0|$ are the cardinality of $I$, $|I_0|$
            respectively  and
          all $\alpha^{\prime}_{\bullet}\in I-I_0$;
    then,
    $$
     {\cal A}\; =\;
      \sum_{I_0\subset I}\,
        \bigcap_{\alpha\in I_0}\,{\cal A}_\alpha e_{I_0} \hat{e}_{I_0}
    $$
    in ${\cal A}\otimes_{{\cal O}_C}{\cal K}_C$;
  \end{itemize}
 give rise to an open cover $\{U_{\alpha}\}_{\alpha}$ of
  $C_{\cal A}:=\boldSpec{\cal A}$
  with $U_{\alpha} :=\boldSpec{\cal A}_{\alpha}$.

\bigskip

\begin{flushleft}
{\bf Morphisms from $C^{\Azscriptsize}$ to ${\Bbb P}^k$.}
\end{flushleft}
Let $Y$ be the projective space over ${\Bbb C}$
 with a fixed standard affine atlas:
 $$
  Y\;=\; {\Bbb P}^k\; =\; \Proj {\Bbb C}[y_0,y_1,\,\cdots\,, y_k]\;
   =\; \cup_{i=0}^k\, V_i\;
   =\; \cup_{i=0}^k\,
         \Spec {\Bbb C}[\mbox{$\frac{y_0}{y_i}$}\,,\,\cdots\,,\,
                        \mbox{$\frac{y_k}{y_i}$}]\,.
 $$
Here
 $y_{\bullet}/y_i$ are treated as formal variables
  with $y_i/y_i =$ the identity $1$ of the ring
  ${\Bbb C}[\mbox{$\frac{y_0}{y_i}$}\,,\,\cdots\,,\,
            \mbox{$\frac{y_k}{y_i}$}]$;
 the gluing
  $V_i \supset V_{ij}:= V_i\cap V_j\stackrel{\sim}{\leftarrow}
       V_{ji}:= V_j\cap V_i \subset V_j$
  of local affine charts is given by
 $$
  \begin{array}{cccccccl}
   {\Bbb C}[\frac{y_0}{y_i}\,,\,\cdots\,,\,\frac{y_k}{y_i}]
    & \hookrightarrow
    & \frac{{\Bbb C}[\frac{y_0}{y_i}\,,\,\cdots\,,\,
                          \frac{y_k}{y_i}\,,\,\frac{y_i}{y_j}]}
           {\left( \frac{y_j}{y_i}\cdot\frac{y_i}{y_j}-1 \right)}
    & \stackrel{\sim}{\longrightarrow}
    & \frac{{\Bbb C}[\frac{y_0}{y_j}\,,\,\cdots\,,\,
                          \frac{y_k}{y_j}\,,\,\frac{y_j}{y_i}]}
           {\left( \frac{y_i}{y_j}\cdot\frac{y_j}{y_i}-1 \right)}
    & \hookleftarrow
    & {\Bbb C}[\frac{y_0}{y_j}\,,\,\cdots\,,\,\frac{y_k}{y_j}]   \\[4.6ex]
   && \frac{y_{\bullet}}{y_i}
    & \longmapsto  & \frac{y_{\bullet}}{y_j}\cdot\frac{y_j}{y_i} \\[1.6ex]
   && \frac{y_i}{y_j}  & \longmapsto  & \frac{y_i}{y_j}   &&&.
  \end{array}
 $$
Then, a morphism $\varphi:C^{\Azscriptsize}\rightarrow {\Bbb P}^k$
 is presented by the following $(k+1)$-tuple of globally-generated
 ${\cal O}_C$-algebra-subsets of
 ${\cal O}_C^{\Azscriptsize}\otimes_{{\cal O}_C}{\cal K}_C$
 from an underlying gluing system of ring-homomorphisms:
 $$
  \left\{\left(\,
   {\cal A}_{(i)}\, :=\,
    {\cal O}_C
     \langle m_{(i),j}\,|\, j\in\{0,\,\cdots\,, k\}-\{i\}\,\rangle
      \;,\;
     e_{(i)}\,
  \right)\right\}_{i=0}^k\,,
 $$
 where (we set $m_{(i),i}=e_{(i)}$ by convention)
 \begin{itemize}
  \item[$\cdot$]
   $\{m_{(i),j}\,|\,i,j=0,\,\ldots\,,k\}$
    is a commuting set of elements in
    $\Gamma(
       {\cal O}_C^{\Azscriptsize}\otimes_{{\cal O}_C}{\cal K}_C)$,

  \item[$\cdot$]
   $e_{(i)}$, $i=0,\,\ldots\,,k$, are admissible idempotent
    pseudo-section of ${\cal O}_C^{\Azscriptsize}$,
 \end{itemize}
 that satisfies the following gluing conditions that match
  with those for the fixed atlas $\{V_i\}_{i=0}^k$ on ${\Bbb P}^k$:
\begin{itemize}
 \item[(1)]
  $1 =\sum_i e_{(i)} -\sum_{i_1<i_2} e_{(i_1)}e_{(i_2)}
     + \,\cdots\,\\[.6ex]
     \mbox{\hspace{6em}}
      + (-1)^{\bullet-1}\sum_{i_1<\,\cdots\,<i_{\bullet}}
                       e_{(i_1)}\,\cdots\,e_{(i_{\bullet})}
      +\,\cdots\,+(-1)^k e_{(0)}\,\cdots\,e_{(k)}$;

 \item[(2)]
  $( m_{(j),j} m_{(i),j} )\,( m_{(j),j} m_{(j),i})\,
   =\, m_{(i),i} m_{(j),j}\,$,
  $\,i,j = 0,\,\ldots\,,k$;

 \item[(3)]
  $m_{(j),j} m_{(i),\bullet}\,
   =\, m_{(j),\bullet} \cdot ( m_{(j),j} m_{(i),j} )\,$,
  $\,i,j, \bullet =0,\,\ldots\,,k$;

 \item[(4)]
  for $I\subset \{0,\,\cdots\,,k\}$, let
  \begin{eqnarray*}
   e_I  & = & \prod_{i\in I}e_i\,,\\[1.2ex]
   \hat{e}_I
    & = & 1 - \sum_j e_{(j)} + \sum_{j_1<j_2} e_{(j_1)}e_{(j_2)}
              + \,\cdots\,\\[.6ex]
    && \mbox{\hspace{4em}}
       + (-1)^l\sum_{j_1<\,\cdots\,<j_l}
                       e_{(j_1)}\,\cdots\,e_{(j_l)}
       +\,\cdots\,
       + (-1)^{k+1-|I|} e_{(j_1)}\,\cdots\, e_{(j_{k+1-|I|})}\,,
  \end{eqnarray*}
  where $|I|$ is the cardinality of $I$ and
        all $j_{\bullet}\in \{0,\,\cdots\,,k\}-I$;
  then, the ${\cal O}_C$-subalgebra
  $$
   {\cal A}\;:=\;
    \sum_{I\subset\{0,\,\cdots\,,k\}}\,
      \bigcap_{i\in I}\,{\cal A}_i e_I \hat{e}_I
  $$
  of ${\cal O}_C^{\Azscriptsize}\otimes_{{\cal O}_C}{\cal K}_C$
  lies in ${\cal O}_C^{\Azscriptsize}$.
\end{itemize}
Note that Condition (1) -- Condition (3) resemble
 the smearing of D0-branes along $C$ (cf.\ [L-Y2: Sec.~4.4])
while the additional Condition (4) says that
 $C_{\varphi}:=\boldSpec{\cal A}$ is dominated by
  the Azumaya noncommutative space $\Space{\cal O}_C^{\Azscriptsize}$,
  should the latter be constructed functorially.
It implies also that $C_{\varphi}$ is proper over $C$.
(These properties from Condition (4) are automatic in the case
 of D$0$-branes).

In terms of this $(k+1)$-tuple of ${\cal O}_C$-algebra-sets,
any gluing system of ring-homomorphisms behind $\varphi$ follows
 from\footnote{Recall
               that a morphism from a (commutative) scheme $X$
               to ${\Bbb A}^k$ is specified by
               a $k$-tuple of elements in the function ring
               $R(X):=\Gamma({\cal O}_X)$ of $X$.}
 $$
  \left\{\,
   \begin{array}{ccccc}
    \varphi^{\sharp}_{(i)}  & :
     & {\Bbb C}[\mbox{$\frac{y_0}{y_i}$}\,,\,\cdots\,,\,
               \mbox{$\frac{y_k}{y_i}$}]\;
     & \longrightarrow   & \Gamma({\cal A}_{(i)})      \\[1.6ex]
    && 1                 & \longmapsto   & e_{(i)}     \\[.6ex]
    && \frac{y_j}{y_i}   & \longmapsto   & m_{(i),j}
   \end{array}\,
  \right\}_{i=0}^k\,.
 $$
The corresponding
 $\varphi_{(i)}:\boldSpec{\cal A}_{(i)}\rightarrow V_i$,
                                             $i=0,\,\ldots\,,k$,
 glue to
 $f_{\varphi}:C_{\varphi}=\boldSpec{\cal A}\rightarrow {\Bbb P}^k$
 that represents $\varphi$.

\begin{definition}
{\bf [nondegenerate/degenerate with respect to standard atlas].}
{\rm
 A morphism $\varphi:C^{\Azscriptsize}\rightarrow {\Bbb P}^k$
   is said to be {\it nondegenerate with respect to}
   a standard affine atlas ${\Bbb P}^k=\cup_{i=0}^k V_i$,
    as above, for ${\Bbb P}^k$
  if the pullback $f_{\varphi}^{-1}(H_i)$ of the hyperplane divisor
    $H_i:= {\Bbb P}^k-V_i$ on ${\Bbb P}^k$ to $C_{\varphi}$
   is a divisor on $C_{\varphi}$ and
   is supported on the smooth locus of
    $(C_{\varphi})_{\redscriptsize}$,\footnote{Recall
                                that $C_{\varphi}$ does not have
                                embedded points.}
   for $i=0,\,\ldots\,,k$.
 $\varphi$ is said to be {\it degenerate with respect to}
   ${\Bbb P}^k=\cup_0^kV_i$
  if $\varphi$ is not non-degenerate with respect to
   ${\Bbb P}^k=\cup_0^kV_i$.
} \end{definition}

In the above discussion, we fix a standard affine atlas on ${\Bbb P}^k$
 and a morphism $\varphi:C^{\Azscriptsize}\rightarrow {\Bbb P}^k$
 can be degenerate with respect to the atlas.
However, given $\varphi:C^{\Azscriptsize}\rightarrow {\Bbb P}^k$,
 one can always choose an atlas,
  still denoted by ${\Bbb P}^k=\cup_0^kV_i$
   with $V_i=
         \Spec {\Bbb C}[\frac{y_0}{y_i}\,,\,\cdots\,,\,\frac{y_r}{y_i}]$,
  so that $\varphi$ becomes nondegenerate with respect to
   ${\Bbb P}^k=\cup_0^kV_i$.
In this case, the above presentation
 $$
  \left\{\left(\,
   {\cal A}_{(i)}\, :=\,
    {\cal O}_C
     \langle m_{(i),j}\,|\, j\in\{0,\,\cdots\,, k\}-\{i\}\,\rangle
      \;,\;
     e_{(i)}\,
  \right)\right\}_{i=0}^k\,,
 $$
 of $\varphi$ has the additional property that:
 \begin{itemize}
  \item[$\cdot$]
   $e_{(i)}=1\in \Gamma({\cal O}_C^{\Azscriptsize})$,
   for $i=0,\,1,\,\ldots,\,k$.
 \end{itemize}
The above presentation of $\varphi$ can then be simplified to
 a commuting $(k+1)$-tuple of rational sections of
 ${\cal O}_C^{\Azscriptsize}$
 $$
   (m_{(0),0},\, m_{(0),1},\,\cdots\,,m_{(0),k})\;  \in\;
   (\Gamma(
    {\cal O}_C^{\Azscriptsize}\otimes_{{\cal O}_C}{\cal K}_C ))^{k+1}
 $$
 that satisfies: ($m_{(0),0}=1$ by convention)
 \begin{itemize}
  \item[(1)]
   $m_{(0),j}$ is generically invertible, for $j=1,\,\ldots\,,k$.

  \item[]
   In this case, $m_{(0),j}$ is invertible in
    $\Gamma({\cal O}_C^{\Azscriptsize}\otimes_{{\cal O}_C}{\cal K}_C)$
   and its inverse therein will be denoted by
    $m_{(0),j}^{-1}$ or $1/m_{(0),j}$.

  \item[(2)]
   Let
    $$
     {\cal A}_{(i)}\; =\;
     {\cal O}_C \langle
      m_{(0),0}/m_{(0),i}, m_{(0),1}/m_{(0),i},\,
                           \cdots\,,m_{(0),k}/m_{(0),i} \rangle\,.
    $$
   Then, the ${\cal O}_C$-subalgebra
    ${\cal A}:=\cap_{i=0}^k {\cal A}_{(i)}$
     of ${\cal O}_C^{\Azscriptsize}\otimes_{{\cal O}_C}{\cal K}_C$
    lies in
     ${\cal O}_C^{\Azscriptsize} \subset
       {\cal O}_C^{\Azscriptsize}\otimes_{{\cal O}_C}{\cal K}_C$.
 \end{itemize}
In terms of this $k$-tuple, any gluing system of ring-homomorphisms
 behind $\varphi$ follows from
 $$
  \left\{\,
   \begin{array}{ccccc}
    \varphi^{\sharp}_{(i)}  & :
     & {\Bbb C}[\mbox{$\frac{y_0}{y_i}$}\,,\,\cdots\,,\,
                \mbox{$\frac{y_k}{y_i}$}]\;
     & \longrightarrow\;
     & \Gamma({\cal O}_C^{\Azscriptsize}\otimes_{{\cal O}_C}{\cal K}_C)
                                               \\[1.6ex]
    && 1                 & \longmapsto   & 1   \\[.6ex]
    && \frac{y_j}{y_i}   & \longmapsto   & m_{(0),j}/m_{(0),i}
   \end{array}\,
  \right\}_{i=0}^k\,.
 $$
The corresponding
 $\varphi_{(i)}:\boldSpec{\cal A}_{(i)}\rightarrow V_i$,
                                            $i=0,\,\ldots\,,k$,
 glue to an
  $f_{\varphi}:C_{\varphi}=\boldSpec{\cal A}\rightarrow {\Bbb P}^k$
 that represents $\varphi$.

\begin{definition}
{\bf [strongly nondegenerate with respect to standard atlas].}
{\rm
 A morphism $\varphi:C^{\Azscriptsize}\rightarrow {\Bbb P}^k$
   is said to be {\it strongly nondegenerate with respect to}
   a standard affine atlas ${\Bbb P}^k=\cup_{i=0}^k V_i$
   for ${\Bbb P}^k$
 if $\varphi$ is nondegenerate with respect to $\{V_i\}_{i=0}^k$ and
     $f_{\varphi}(C_{\varphi})
     \subset {\Bbb P}^k-\cup_{0\le i<j\le k}(H_i\cap H_j)$.
} \end{definition}

Realize $(y_0,\,\cdots\,,y_k)$
 as a $(k+1)$-tuple ${\mathbf t} :=(t_0,\,\cdots\,t_k)$ of
 global sections of ${\cal O}_{{\Bbb P}^k}(1)$
 with $t_i^{-1}(0)=H_i$.
Then, ${\mathbf t}$ determines an embedding
 $$
  \begin{array}{cccccl}
   {\mathbf \pi}\,=\,(\pi_{(ij)})_{0\le i<j\le k}
    & : & {\Bbb P}^k-\cup_{0\le i<j\le k}(H_i\cap H_j)
    & \longrightarrow
    & \prod_{0\le i<j\le k}{\Bbb P}^1_{(ij)}\,
      \simeq\, ({\Bbb P}^1)^{k(k+1)/2}            \\[1.2ex]
   && [y_0:\,\cdots\,:y_k]  & \longmapsto
    & (y_j/y_i)_{i<j}\; =\; (t_j/t_i)_{i<j}  &,
  \end{array}
 $$
 where ${\Bbb P}^1_{(i),j}\simeq {\Bbb P}^1$
  with a fixed coordinate ${\Bbb C}\cup\{\infty\}$.
When $\varphi: C^{\Azscriptsize}\rightarrow {\Bbb P}^k$ is strongly
 nondegenerate with respect to $\{V_i\}_{i=0}^k$,
the presentation
 $(m_{(0),1},\,\cdots\,,m_{(0),k})\,\in\,
  (\Gamma(
     {\cal O}_C^{\Azscriptsize}\otimes_{{\cal O}_C}{\cal K}_C ))^k$
 of $\varphi$ with respect to the above data is characterized by
 (recall $m_{(0),0}=1$ by convention)
 \begin{itemize}
  \item[(1)]
   $m_{(0),j}$ is generically invertible and
   $m_{(0),j_1}m_{(0),j_2}=m_{(0),j_2}m_{(0),j_1}$,
   for $j, j_1, j_2 = 1,\,\ldots\,,k$;

  \item[(2)]
   ${\cal O}_C\langle m_{(0),j}/m_{(0),i}\rangle
    \cap {\cal O}_C\langle m_{(0),i}/m_{(0),j}\rangle
    \subset {\cal O}_C^{\Azscriptsize}$.
 \end{itemize}
 It follows from the minimal property of $C_{\varphi}$ that
  $$
   {\cal A}_{\varphi}\;
   =\; {\cal O}_C
       \left\langle
         {\cal O}_C\langle m_{(0),j}/m_{(0),i}\rangle
         \cap {\cal O}_C\langle m_{(0),i}/m_{(0),j}\rangle\:
                                 :\: 0\le i<j\le k \right\rangle
  $$
  in such a presentation of $\varphi$.

\begin{remark} {\it $[$geometry of commutative ${\cal O}_C$-subalgebra
    of ${\cal O}_C^{\Azscriptsize}\otimes_{{\cal O}_C}{\cal K}_C$$]$.}
{\rm
 Let
  $a$ be a non-nilpotent rational section of
   ${\cal O}_C^{\Azscriptsize}$  and
  ${\cal O}_C\langle a \rangle$ be the ${\cal O}_C$-subalgebra
   of ${\cal O}_C^{\Azscriptsize}\otimes_{{\cal O}_C}{\cal K}_C$
   generated by\footnote{Recall
                          that $1$ is always included in
                          ${\cal O}_C\langle\,\cdot\,\rangle$
                          by convention.}
   $a$.
 Then, $C_a := \boldSpec({\cal O}_C\langle a\rangle)$
  is an open dense subscheme of a finite scheme $\overline{C}_a$
  over $C$ with a tautological rational function $h_a$ on $C_a$
   (or equivalently on $\overline{C}_a$) corresponding to $a$.
 The Generalized Cayley-Hamilton Theorem for a linear operator
  on a finite-dimensional vector space implies that
  $C_a$ is a subscheme of the spectral curve
  $\det(\lambda-a)=0$ in $C\times \Spec {\Bbb C}[\lambda]$
  with the same reduced-scheme structure.
 This implies that $C_a/C$ is canonically a closed subscheme
  of $C\times{\Bbb A}^1$.
 Note that
  $C_a$ contains an open dense subset that is finite
   over an open subset of $C$.
 However, $C_a$ itself in general may not be finite over
  every open subset of $C$.
 Cf.~Figure~2-4-2.
 In general, let ${\cal B}$ be a commutative ${\cal O}_C$-subalgebra
  of ${\cal O}_C^{\Azscriptsize}\otimes_{{\cal O}_C}{\cal K}_C$
  that is globally generated by
  $a_1,\,\ldots\,, a_l \in
   \Gamma({\cal O}_C^{\Azscriptsize}\otimes_{{\cal O}_C}{\cal K}_C)$;
  then $C_{\cal B}:=\boldSpec{\cal B}$ is a closed subscheme of
   the fibered product $C_{a_1}\times_C\,\cdots\,\times_C C_{a_l}$
   and hence is a closed subscheme of $C\times {\Bbb A}^l$.
 Using this, one obtains a second proof of Lemma/Definition~2.1.8.
} \end{remark}
\begin{figure}[htbp]
 \epsfig{figure=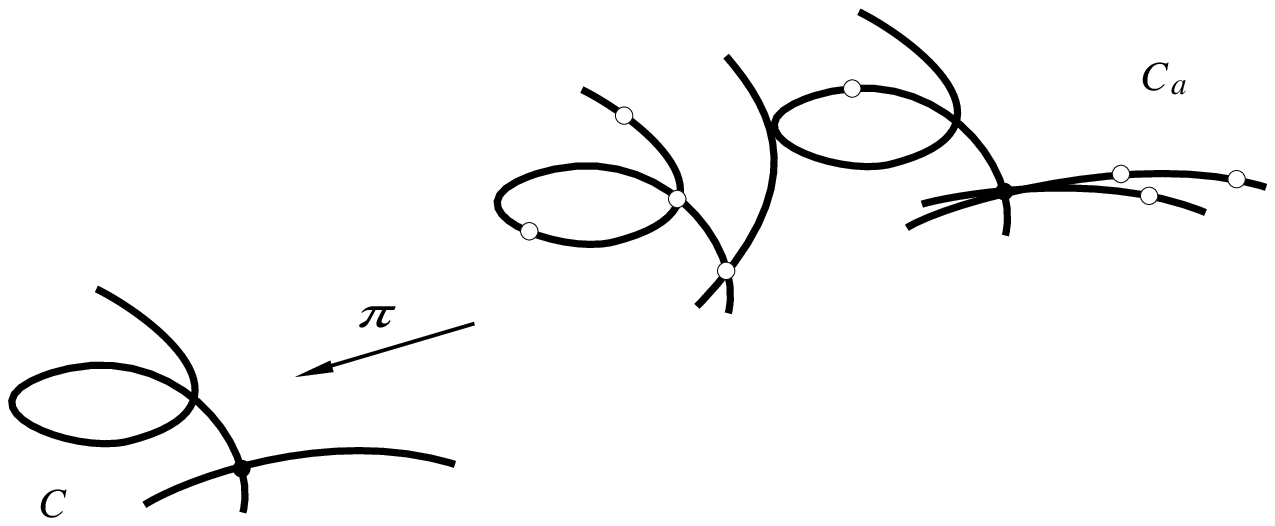,width=16cm}
 \centerline{\parbox{13cm}{\small\baselineskip 12pt
  \center{{\sc Figure} 2-4-2.
  The geometry of $C_a$.
  } }}
\end{figure}

\bigskip

\section{The stack ${\frak M}_{A\!z(g,r,\chi)^f}(Y,\beta)$
         of morphisms with a fundamental module.}

\subsection{D0-branes revisited: the stack ${\frak M}^{D0}_r(Y)$
            of D0-branes of type $r$ on $Y$.}

The moduli stack ${\frak M}^{D0}_r(Y)$ of D$0$-branes of type $r$
 on $Y$ under the Polchinski-Grothendieck Ansatz
 is the moduli stack of morphisms
 from the (un-fixed) Azumaya point $\Space(M_n({\Bbb C}))$
  with the fundamental module ${\Bbb C}^r$ to $Y$.
It follows from Lemma~2.2.1 and Remark~2.2.2 that
 this is precisely the stack of $0$-dimensional, length $r$,
 torsion sheaves on $Y$.
Such an ${\cal O}_Y$-module has the Hilbert polynomial
 the constant $r$,
 in disregard of the polarization ${\cal O}_Y(1)$ on $Y$ chosen.
Such a torsion sheaf is thus a quotient sheaf of
 ${\cal O}_Y\otimes{\Bbb C}^r$.
It follows that ${\frak M}^{D0}_r(Y)$ is a quotient stack
 of the Quot-scheme of ${\cal O}_Y\otimes{\Bbb C}^r$
  with the constant Hilbert polynomial $r$:
 $$
  {\frak M}^{D0}_r(Y)\;
   =\; [\Quot_r({\cal O}_Y\otimes{\Bbb C}^r)/\GL(r;{\Bbb C})]\,.
 $$

There are three substacks of ${\frak M}^{D0}_r(Y)$ worth mentioning:
 \begin{itemize}
  \item[(1)]
   ${\frak M}^{D0}_{r,\scriptsizeChow}(Y)$
    consists of objects
     $[\widetilde{\cal E}]\in {\frak M}^{D0}_r(Y)$
    such that $\Supp\widetilde{\cal E}$ is reduced;
   it fibers over the stack $[Y^r/\Sym_r]$ of $0$-cycles of order $r$
    on $Y$.

  \item[(2)]
   ${\frak M}^{D0}_{r,\scriptsizeHilb}(Y)$
    consists of objects
     $[\widetilde{\cal E}]\in {\frak M}^{D0}_r(Y)$
    such that
     ${\cal O}_{\scriptsizeSupp\widetilde{\cal E}}
                             \simeq \widetilde{\cal E}$;
   it fibers over the Hilbert scheme $\Hilb_r(Y)$ of $0$-dimensional
    subschemes of length $r$ on $Y$.

  \item[(3)]
   ${\frak M}^{D0}_{r, \singletonscriptsize}(Y)$
    consists of objects
     $[\widetilde{\cal E}]\in {\frak M}^{D0}_r(Y)$
    such that
     $(\Supp(\widetilde{\cal E}))_{\redscriptsize}$
     is a singleton;
   it fibers over $Y$.
 \end{itemize}
The role of these three substacks in superstring theory
 are respectively:
 \begin{itemize}
  \item[$(1^{\prime})$]
   ${\frak M}^{D0}_{r,\scriptsizeChow}(Y)$
    is related to the classical moduli space of
    gas of $r$-many D$0$-branes on $Y$.

  \item[$(2^{\prime})$]
   ${\frak M}^{D0}_{r,\scriptsizeHilb}(Y)$
    is related to the quantum moduli space of
    gas of $r$-many D$0$-branes on $Y$, when $Y$ is a smooth
    projective surface,

  \item[$(3^{\prime})$]
   ${\frak M}^{D0}_{r,\singletonscriptsize}(Y)$
    could contain various birational models for $Y$
     and, when $Y$ is singular, partial smooth resolutions of $Y$.
   These partial resolutions are the candidate target-space structures
    revealed as the vacuum manifolds/varieties/spaces
    in the different phases of the quantum field theory
    on the $r$-stacked D$0$-brane probe to a singular $Y$.
 \end{itemize}
See [Va] for Item $(1^{\prime})$ and Item $(2^{\prime})$ and
 [L-Y2: Sec.~2.1, Sec.~4.4, and quoted references].

It follows from Corollary~2.2.3 that
 one can construct ${\frak M}_{\Azscriptsize(g,r,\chi)^f}(Y,\beta)$
 as a substack of the stack of morphisms
 from prestable curves of genus $g$ to ${\frak M}^{D0}_r(Y)$.
Cf.\ Figure~3-1-1.

\begin{figure}[htbp]
 \epsfig{figure=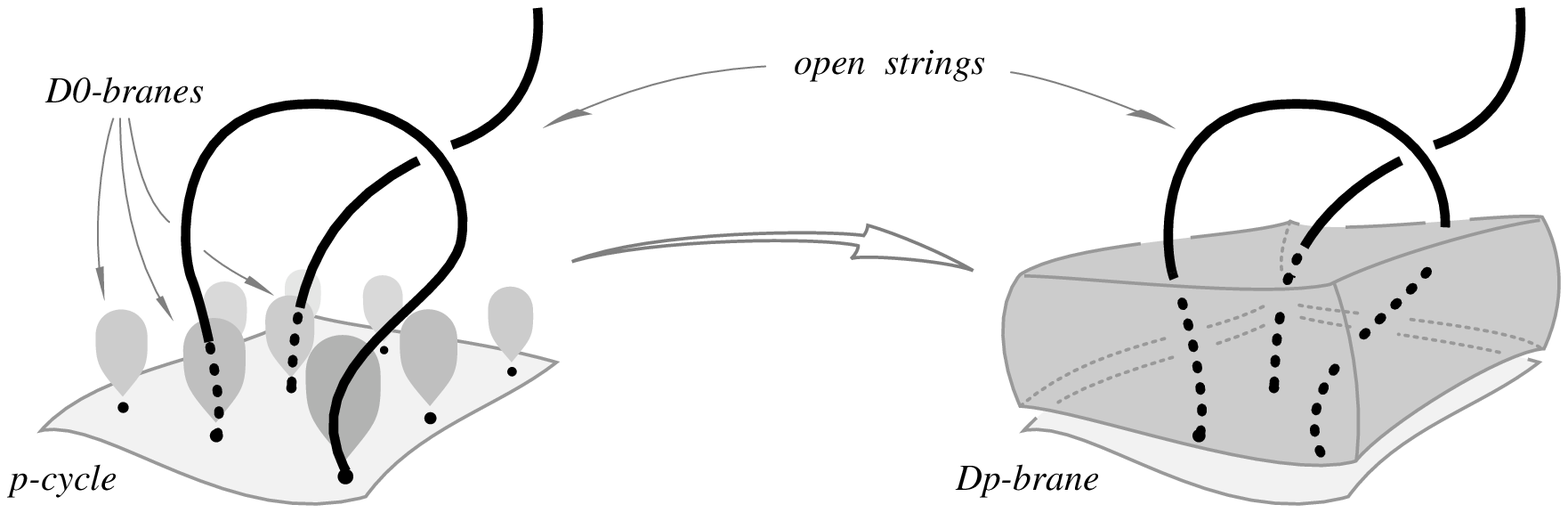,width=16cm}
 \centerline{\parbox{13cm}{\small\baselineskip 12pt
  {\sc Figure} 3-1-1.
  The original stringy operational definition of D-branes
   as objects in the target space(-time) $Y$ of strings
   where end-points of open-strings can and have to stay
  suggests that smearing D$0$-branes along
   a (real) $p$-dimensional cycle/submanifold/distribution $Z$
   in $Y$ renders $Z$ a D$p$-brane.
 Such a smearing in our case is realized as
  a morphism from a manifold/variety $Z$ to the stack
  ${\frak M}^{D0}(Y)$ of D$0$-branes on $Y$.
 In the figure, the Chan-Paton sheaf ${\cal E}$ that carries
  the index information on the end-points of open strings
  is indicated by a shaded cloud.
 Its endomorphism sheaf $\Endsheaf{\cal E}$ carries
  the information of the gauge group of the quantum field theory
  on the D-brane world-volume.
  }}
\end{figure}

\bigskip

\subsection{The stack ${\frak M}_{A\!z(g,r,\chi)^f}(Y,\beta)$
            and its decorated atlas.}

The forgetful functor that forgets sheaves of modules on products
 of prestable curves with $Y$ give rise to a fibration/morphism:
 $$
  \pi_{{\frak M}_g}\;:\;
   {\frak M}_{\Azscriptsize(g,r,\chi)^f}(Y,\beta)\;
   \longrightarrow\; {\frak M}_g\,,
 $$
 where
  ${\frak M}_g$ is the stack of prestable curves of genus $g$.
 It is known that ${\frak M}_g$ is algebraic (Artinian in this case),
  smooth of stacky dimension $3g-3$, but only locally of finite type.
That ${\frak M}_{\Azscriptsize(g,r,\chi)^f}(Y,\beta)$ is algebraic
 (Artinian in this case) follows from the fact that
   both ${\frak M}_g$ and the fiber stacks $\Cohfrak(C\times Y)$
   of coherent ${\cal O}_{C\times Y}$-modules are algebraic.
An atlas of ${\frak M}_{\Azscriptsize(g,r,\chi)^f}(Y,\beta)$
 can be constructed as follows, combining
  the morphism $\pi_{{\frak M}_g}$ and
  the existing construction of atlas
   for ${\frak M}_g$ ([Kn]) and for $\Cohfrak(C\times Y)$ ([L-MB])
  respectively.

\bigskip

Fix an $e\ge 3$.
Then,
 for any stable curve with $n$ marked points $(C;p_1,\,\cdots,\,p_n)$,
 the line bundle $(\omega_C(p_1+\,\cdots\,+p_n))^{\otimes e}$,
  where $\omega_C$ is the dualizing sheaf of $C$,
 is very ample.
Let
 \begin{eqnarray*}
  d  & =
     & e(2g-2+n)\;  =\;
       \degree((\omega_C(p_1+\,\cdots\,+p_n))^{\otimes e})\,,  \\[.6ex]
  P  & =
     & \mbox{Hilbert polynomial of
             $(\omega_C(p_1+\,\cdots\,+p_n))^{\otimes e}$}\,:\\
     && P(m)\; =\; dm+(1-g)\,,  \\[.6ex]
  N  & =  & P(1)-1 \;=\; d-g\,.
 \end{eqnarray*}
Let
 $$
  H_{g,n}\; \subset\; \Hilb_P({\Bbb P}^N)\times ({\Bbb P}^N)^n
 $$
 be the locally closed subscheme
 that parameterizes\footnote{It
                              is standard to write both $H_{g,n}$ here
                               and
                              {\it Quot}$^{\,\circ}
                                         _{(r,\chi_{n^{\prime}})}
                               ({\cal O}_{(C_{g,n}\times Y)/H_{g,n}}
                                 \otimes {\Bbb C}^{N^{\prime}})$
                              in the next more precisely as schemes
                              that represent related moduli functors.}
  tuples $([C], p_1,\,\cdots\,,p_n)$
  where
   $C$ is a subscheme of ${\Bbb P}^N$ of Hilbert polynomial $P$
     that is nodal,
   $p_i\in C$ for $i=1\,\ldots\,,n$, and
   ${\cal O}_{{\Bbb P}^N}(1)|_{C}
    \simeq (\omega_C(p_1+\,\cdots\,p_n))^{\otimes e}$.
Then,
 the morphism $\pi_{g,n}: H_{g,n}\rightarrow {\frak M}_g$
  defined by the tautological family $C_{g,n}$
  of nodal curves over $H_{g,n}$
 is smooth,
  with image an open substack of ${\frak M}_g$ of finite type,
 and
the morphism
 $$
  \pi_g\,:=\; \amalg_n \pi_{g,n}\; :\;
   \amalg_n H_{g,n}\;\longrightarrow\; {\frak M}_g
 $$
 is a smooth epimorphism and defines an atlas on ${\frak M}_g$
 in the fppf topology.

Furthermore, by construction,
 the universal subscheme $C_{g,n}$ over $H_{g,n}$
 is naturally endowed with a relative very ample line bundle, namely
 ${\cal O}_{C_{g,n}/H_{g,n}}(1)
   := \iota_{g,n}^{\ast}{\cal O}_{{\Bbb P}^N}(1)$,
  where $\iota_{g,n}:C_{g,n}\rightarrow {\Bbb P}^N$
  is the tautological morphism.
We will call
 $\{ (H_{g,n}, \pi_{g,n}; {\cal O}_{C_{g,n}/H_{g,n}}(1)) \}_n$
 a {\it decorated atlas}\footnote{It
         is conceptually instructive to think of the system
          $\{ {\cal O}_{C_{g,n}/H_{g,n}}(1) \}_n$
          as defining a {\it twisted relative (very) ample line bundle}
          on the universal curve ${\frak C}_g$ over ${\mathfrak M}_g$
         since there is no relative ample line bundle
          on the whole ${\frak C}_g/{\frak M}_g$.
         Similarly, for the system
          $\{ {\cal O}_{(C_{g,n}\times Y)/H_{g,n}}(1) \}
                                         _{n; N^{\prime}, n^{\prime}}$
         in the next.}
 on ${\frak M}_g$.
By construction,
 ${\cal O}_{C_{g,n}/H_{g,n}}(1)$ has relative degree $d=e(2g-2+n)$.

\bigskip

Fix an ample line bundle ${\cal O}_X(1)$ on $X$ and let
 $H_Y\in A^1(X)$ be the hyperplane class associated to ${\cal O}_Y(1)$.
Then,
 ${\cal O}_{(C_{g,n}\times Y)/H_{g,n}}(1)
  := {\cal O}_{C_{g,n}/H_{g,n}}(1) \boxtimes {\cal O}_X(1)$
 is a relative ample line bundle on $(C_{g,n}\times Y)/H_{g,n}$.
Define
 $$
  \chi_{n^{\prime}}\;
   :=\; \chi +
         n^{\prime}
          \left(\rule{0em}{1em}
           \reldegree({\cal O}_{C_{g,n}/H_{g,n}}(1)) + H_Y\cdot\beta
          \right)\;
    =\; \chi +
         n^{\prime} \left(\rule{0em}{1.2em}
                          re(2g-2+n) + H_Y\cdot\beta \right)\,.
 $$
Let
 $\Quot_{(r,\chi_{n^{\prime}})}^{\,\circ}
   ({\cal O}_{(C_{g,n}\times Y)/H_{g,n}}\otimes{\Bbb C}^{N^{\prime}})$
 be the open subscheme of the Quot-scheme\\
  $\Quot({\cal O}_{(C_{g,n}\times Y)/H_{g,n}}\otimes{\Bbb C}^{N^{\prime}})$
 that parameterizes isomorphism classes of
  quotient ${\cal O}_{C\times Y}$-modules
  ${\cal O}_{C\times Y}\otimes{\Bbb C}^{N^{\prime}}
     \rightarrow \widetilde{\cal E} \rightarrow 0$
  on fibers $C\times Y$ of $(C_{g,n}\times Y)/H_{g,n}$
  such that
  \begin{itemize}
   \item[$\cdot$]
    $\widetilde{\cal E}$ is of relative length $r$
      on $(C\times Y)/C$;

   \item[$\cdot$]
    the Hilbert polynomial
    $P(\widetilde{\cal E}, m)
     =  \left(\rule{0em}{1em}
            re(2g-2+n) + H_Y\cdot\beta \right)\,m
        + \chi_{n^{\prime}}$;

   \item[$\cdot$]
    $H^i(C\times Y, \widetilde{\cal E})=0$, for $i>0$;

   \item[$\cdot$]
    the natural map
     ${\Bbb C}^{N^{\prime}}
      \rightarrow H^0(C\times Y, \widetilde{\cal E})$
     is an isomorphism.
  \end{itemize}
Then,
the universal quotient sheaf twisted by the pull-back of
 ${\cal O}_{(C_{g,n}\times Y)/H_{g,n}}(-n^{\prime})$
 $$
 \begin{array}{l}
  {\cal O}
   _{ \scriptsizeQuot_{(r,\chi_{n^{\prime}})}^{\,\circ}
                ({\cal O}_{C_{g,n}\times Y}\otimes{\Bbb C}^{N^{\prime}})
      \times_{H_{g,n}} (C_{g,n}\times Y) }
    \otimes
    {\cal O}_{(C_{g,n}\times Y )/H_{g,n}}(-n^{\prime})
    \otimes {\Bbb C}^{N^{\prime}}  \\[.6ex]
  \hspace{20em}
   \longrightarrow\;
   \widetilde{\cal F}
    \otimes
    {\cal O}_{(C_{g,n}\times Y )/H_{g,n}}(-n^{\prime})\;
   \longrightarrow\; 0
 \end{array}
 $$
 on
 $\Quot_{(r,\chi_{n^{\prime}})}^{\,\circ}
              ({\cal O}_{C_{g,n}\times Y}\otimes{\Bbb C}^{N^{\prime}})
   \times_{H_{g,n}} (C_{g,n}\times Y)$,
 for $n^{\prime}\ge 0$,
 defines a smooth morphism
 $$
  \Theta_{(n;N^{\prime},n^{\prime})}\; :\;
    \Quot_{(r,\chi_{n^{\prime}})}^{\,\circ}
           ({\cal O}_{C_{g,n}\times Y}\otimes{\Bbb C}^{N^{\prime}})\;
  \longrightarrow\;
  {\frak M}_{\Azscriptsize(g,r,\chi)^f}(Y,\beta)\,.
 $$
The morphism from their union
 $$
  \Theta\; :=\;
   \amalg_{n, N^{\prime}, n^{\prime}\ge 0}\,
    \Theta_{(n; N^{\prime}, n^{\prime})}\; : \;
   \amalg_{n, N^{\prime}, n^{\prime}\ge 0}\,
    \Quot_{(r,\chi_{n^{\prime}})}^{\,\circ}
           ({\cal O}_{C_{g,n}\times Y}\otimes{\Bbb C}^{N^{\prime}})\;
  \longrightarrow\;
  {\frak M}_{\Azscriptsize(g,r,\chi)^f}(Y,\beta)
 $$
 is a smooth epimorphism and defines an atlas on
 ${\frak M}_{\Azscriptsize(g,r,\chi)^f}(Y,\beta)$
 in the fppf topology.
We will call
 $$
  \left\{\,
   \left(
    \Quot_{(r,\chi_{n^{\prime}})}^{\,\circ}
           ({\cal O}_{C_{g,n}\times Y}\otimes{\Bbb C}^{N^{\prime}})\,,
     \Theta_{(n;N^{\prime},n^{\prime})}\,;\,
     {\cal O}_{(C_{g,n}\times Y)/H_{g,n}}(1)
   \right)
  \right\}_{n;N^{\prime},n^{\prime}}
 $$
 a {\it decorated atlas}
 on ${\frak M}_{\Azscriptsize(g,r,\chi)^f}(Y,\beta)$,
 where ${\cal O}_{(C_{g,n}\times Y)/H_{g,n}}(1)$
  is now regarded as a relative ample line bundle on
  $(\Quot_{(r,\chi_{n^{\prime}})}^{\,\circ}
      ({\cal O}_{C_{g,n}\times Y}\otimes{\Bbb C}^{N^{\prime}})
    \times_{H_{g,n}} (C_{g,n}\times Y))
   /\Quot_{(r,\chi_{n^{\prime}})}^{\,\circ}
      ({\cal O}_{C_{g,n}\times Y}\otimes{\Bbb C}^{N^{\prime}})$.

\bigskip

\section{Remarks: D-strings as a master object for curves and
                  D-string world-sheet instantons.}

The moduli stack ${\frak M}_{\Azscriptsize(g,r,\chi)^f}(Y,\beta)$
 will serve as a prototype moduli stack
 to address topological D-strings in this project.
It has its own deformation-obstruction theory
 phrasable in three different aspects.

The content of this stack already renders it
 as a model for a master object for curves:
Via various forgetful functors
 - applied to appropriate substacks and
   with perhaps additional stabilizations -,
 ${\frak M}_{\Azscriptsize(g,r,\chi)^f}(Y,\beta)$ can be related to
 the moduli stack $\overline{M}_g$ of stable curves and
 the moduli stack $\Bun_{(g,r,\chi)}$ of bundles over prestable curves,
  and
 in the genus $0$ case, taking surrogates associated to morphisms
 relates ${\frak M}_{\Azscriptsize(g,r,\chi)^f}(Y,\beta)$
 to Hurwitz schemes.
It is more technical to address how these standard moduli stacks
 for curves can be embedded back into
 ${\frak M}_{\Azscriptsize(g,r,\chi)^f}(Y,\beta)$,
 after allowing reasonable ambiguities.
These ring with an anticipation explained
 in [L-Y2: Introduction and Sec.~4.5],
 which says that
  the moduli space/stack of D-branes should serve
  as a universal/master moduli space/stack for commutative geometry
  from the aspect of
  the Wilson's theory-space picture of stringy dualities,
cf.~Diagram~4-1 (i.e.\ [L-Y1: Diagram A-1-1]) and Figure~4-2.

The notion of (semi-)stability of objects
 in the current moduli problem comes from three sources/origins
 naturally built into the problem.
Choices of this notion define special fillable substacks of finite type
 in ${\frak M}_{\Azscriptsize(g,r,\chi)^f}(Y,\beta)$.
The stack ${\frak M}_{\Azscriptsize(g,r,\chi)^f}(Y,\beta)$
 is part of the ingredients toward a mathematical theory of
 topological D-string world-sheet instanton invariants of $Y$
 as a weighted counting of maps
 from Azumaya prestable curves with a fundamental module to $Y$,
 along the line of the Polchinski-Grothendieck Ansatz.

\begin{figure}[htbp]
 \epsfig{figure=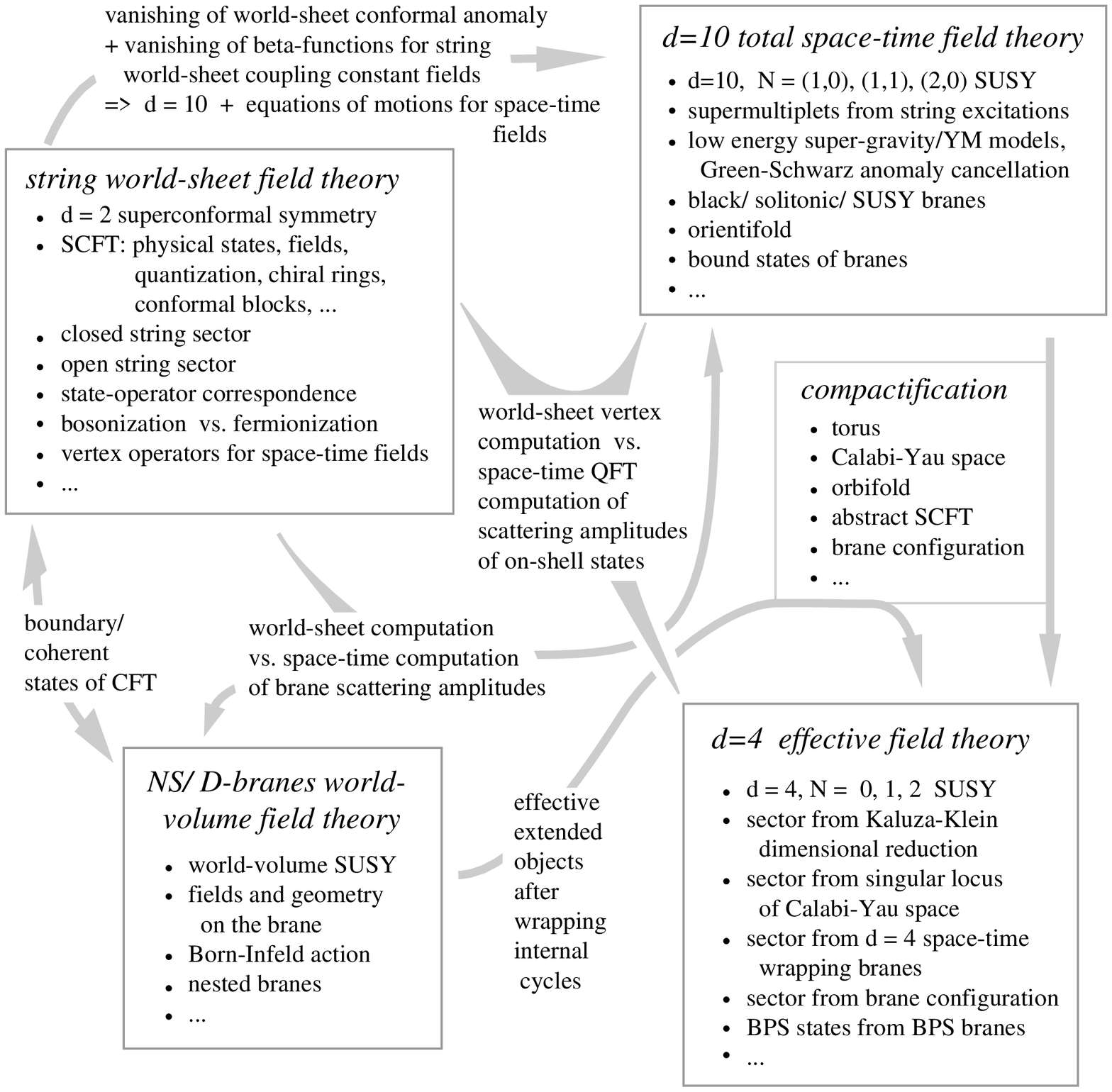,width=16cm}
 \centerline{\parbox{13cm}{\small\baselineskip 12pt
  {\sc Diagram} 4-1. ([L-Y1: Diagram A-1-1].)
  The {\it string world-sheet}, the {\it brane-world-volume},
   the {\it total space-time}, and
   the {\it $4$-dimensional effective field theory}
   aspect of a string theory.
  Scattering amplitudes of fields and D-branes in space-time
   computed via string world-sheet methods and via space-time
   field theory method have to match order by order.
  Each of the four aspects itself has a {\it Wilson's theory-space}
   associated to it, containing all the phases of the field theory
   associated to that aspect of the string theory.
  {\it Dualities} can be realized either as a local isomorphism
   or a coordinate change on the Wilson's theory-space
    ${\cal S}_{\Wilson}$
   that induces an isomorphism on the universal family of
    Hilbert spaces of states, ring of operators, and
    correlation functions of these operators over ${\cal S}_{\Wilson}$.
  In the weak string coupling regime, strings are light and branes
   are heavy and hence string are regarded as more fundamental.
  In the strong string coupling regime, branes can become light and
   strings become heavy and should be no longer treated as
   the most/unique fundamental object in the theory.
  }}
\end{figure}

$ $
\vspace{6em}
\begin{figure}[htbp]
 \epsfig{figure=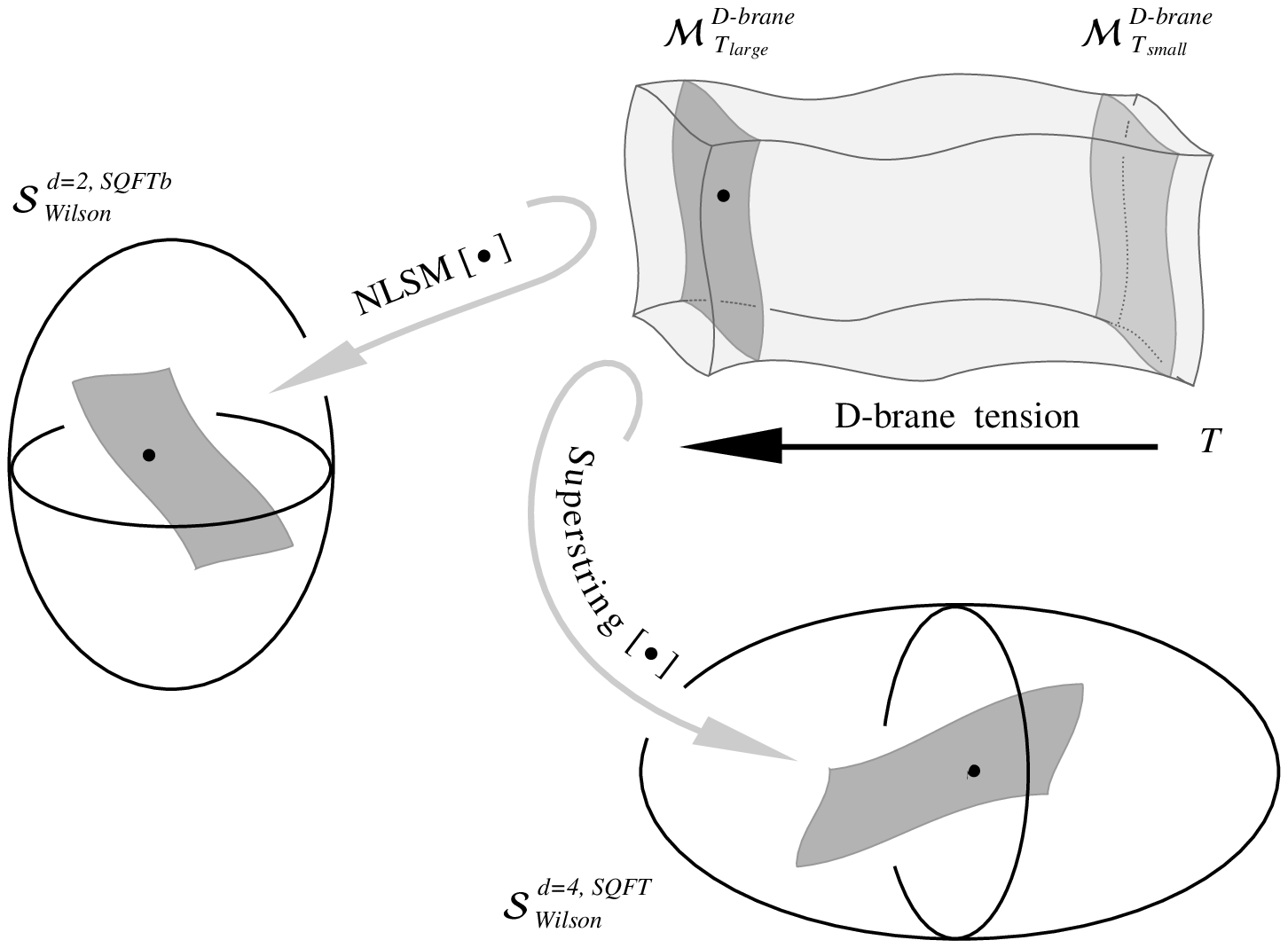,width=16cm}
 \centerline{\parbox{13cm}{\small\baselineskip 12pt
  {\sc Figure} 4-2.
  The moduli space of soft/fundamental D-branes and the
   hard/solitonic D-branes are connected by a continuum
   fibered over the brane-tension $T$-parameter line.
  The moduli space of solitonic D-branes is then embedded
   into the Wilson's theory-space of
    either $d=4$ supersymmetric quantum field theories (SQFT)
    or     $d=2$ supersymmetric quantum field theories-with-boundary
                 (SQFTb)
   via respectively
    the Kaluza-Klein compactification {\it Superstring}$\,[\,\bullet\,]$
    of a $d=10$ superstring model
    or taking nonlinear sigma model-with-boundary
       {\it NLSM}$\,[\,\bullet\,]$
       on the internal geometry-with-D-branes of such a compactification.
  Stringy dualities that arise from D-brane mechanisms
   render then the moduli space of D-branes a master object
   that relates different standard moduli spaces
   from moduli problems in commutative geometry.
  The study of D0-branes in [L-Y2] and D-strings in the current work
   following the Polchinski-Grothendieck Ansatz
   is consistent with such a stringy picture/anticipation.
  }}
\end{figure}

\newpage

\end{document}